\documentclass[reqno, a4paper, 10pt]{amsart}
\usepackage{amssymb,url, color, mathrsfs}
\usepackage[colorlinks=true, bookmarks=true, pdfstartview=FitH, pagebackref=true, linktocpage=true]{hyperref}
\usepackage[short,nodayofweek]{datetime}
\usepackage[pagewise,running,mathlines,displaymath, switch]{lineno}
\usepackage{mathabx}
\usepackage{indentfirst, tabularx}
\usepackage[table]{xcolor}
\usepackage{float}

\usepackage[T5]{fontenc}
\usepackage[utf8]{inputenc}

\theoremstyle{plain}
\numberwithin{equation}{section}
\newtheorem{theorem}{Theorem}[section] 
\newtheorem{lemma}{Lemma}[section] 

\newtheorem{proposition}{Proposition}[section]

\newtheorem{remark}{Remark}[section]

\definecolor{brown}{rgb}{0.5,0,0}
\definecolor{backgroundcolor}{rgb}{0.98, 0.92, 0.73}

\def\R{\mathbf R}
\def\C{\mathbb C}
\def\bB{\mathcal B} 
\def\cM{\mathcal M}
\def\cL{\mathcal L} 
\def\cE{\mathcal E} 
\def\bT{\mathbb T} 
\def\Csf{\mathsf c} 

\newif\ifprint
\ifprint
\else
\fi

\author[Tiến-Tài Nguyễn]{Tiến-Tài Nguyễn }
\address[Tiến-Tài Nguyễn]{Laboratoire Analyse G\'eom\'etrie et Applications, Universit\'e Sorbonne Paris Nord,  93430 - Villetaneuse, France}
\email{\href{mailto: Tiến-Tài Nguyễn <tientai.nguyen@math.univ-paris13.fr>}{tientai.nguyen@math.univ-paris13.fr}}

\begin{document}
\allowdisplaybreaks

\setpagewiselinenumbers
\setlength\linenumbersep{100pt}

\title[Rayleigh-Taylor instability]{Linear and nonlinear analysis of the viscous Rayleigh-Taylor system with Navier-slip boundary conditions}

\begin{abstract}
In this paper, we are interested in the linear and the nonlinear Rayleigh-Taylor instability for the  gravity-driven incompressible Navier-Stokes equations with Navier-slip boundary conditions around a smooth increasing density profile $\rho_0(x_2)$ in a slab domain $2\pi L\bT \times (-1,1)$ ($L>0$, $\bT$ is the usual 1D torus).  The linear instability study of the viscous Rayleigh-Taylor model amounts to the study of the following ordinary differential equation  on the finite interval  $(-1,1)$, 
\begin{equation}\label{EqMain}
-\lambda^2 [ \rho_0 k^2 \phi - (\rho_0 \phi')'] = \lambda \mu (\phi^{(4)} - 2k^2 \phi'' + k^4 \phi) - gk^2 \rho_0'\phi,
\end{equation} with the boundary conditions 
\begin{equation}\label{4thBound}
\begin{cases}
\phi(-1)=\phi(1)=0,\\
\mu \phi''(1) = \xi_+ \phi'(1), \\
\mu \phi''(-1) =- \xi_- \phi'(-1),
\end{cases}
\end{equation}
where  $\lambda>0$ is the growth rate in time, $g>0$ is the gravity constant, $k$ is the wave number  and two Navier-slip coefficients  $\xi_{\pm}$ are   nonnegative constants.  For each $k\in L^{-1}\mathbb{Z}\setminus\{0\}$, we  define a threshold of viscosity coefficient $\mu_c(k,\Xi)$ for linear instability. So that, in the $k$-supercritical regime, i.e. $\mu>\mu_c(k,\Xi)$, we  describe a spectral analysis adapting the operator method initiated by Lafitte-Nguyễn  \cite{LN20} and then prove that there are infinite  nontrivial solutions $(\lambda_n, \phi_n)_{n\geq 1} $ of \eqref{EqMain}-\eqref{4thBound} with $\lambda_n \to 0$ as $n\to \infty$ and $\phi_n\in H^4((-1,1))$.   Based on the existence of infinitely many normal  modes of the linearized problem, we construct  a wide class of initial data to the nonlinear equations, extending the previous framework of Guo-Strauss \cite{GS95} and of Grenier \cite{Gre00}, to prove the nonlinear Rayleigh-Taylor instability in a high regime of viscosity coefficient, namely $\mu >3\sup_{k\in L^{-1}\mathbb{Z}\setminus\{0\}}\mu_c(k,\Xi)$.
\end{abstract}

\date{\bf \today \; at \, \currenttime}



\maketitle

\tableofcontents

\section{Introduction}

The Rayleigh--Taylor (RT) instability, studied first by Lord Rayleigh in \cite{Str83} and then Taylor \cite{Tay50} is well known as a gravity-driven instability in two semi-infinite inviscid and incompressible  fluids when the heavy one is on top of the light one. It  has attracted much attention due to both its physical and mathematical importance. Two applications  worth mentioning are implosion of inertial confinement fusion capsules \cite{Lin98} and core-collapse of supernovae  \cite{Rem00}.  For a detailed physical comprehension of the RT instability, we refer to three survey papers  \cite{Kull91, Zhou17_1, Zhou17_2}. Mathematically speaking,  for the inviscid and incompressible regime with a smooth density profile, the classical  RT instability was investigated by Lafitte \cite{Laf01}, by Guo and Hwang \cite{GH03} and by  Helffer and Lafitte \cite{HL03}.

Concerning the viscous RT instability, one of the first studies can be found in  the book of Chandrasekhar \cite[Chap. X]{Cha61}. He considered two uniform viscous fluid separated by a horizontal boundary and  generalized the classical result of Rayleigh and Taylor. We refer the readers to   mathematical viscous RT studies for  two compressible channel flows by Guo and Tice \cite{GT11}, for incompressible fluid in the whole space $\R^3$ by Jiang et. al   \cite{JJN13} and Lafitte and Nguyễn \cite{LN20}, respectively.

In this paper, we are concerned with the viscous RT of the nonhomogeneous incompressible Navier-Stokes equations with gravity in a 2D slab domain $\Omega= 2\pi L \mathbb{T} \times (-1,1)$ with $L>0$ and $\mathbb{T}$ is the 1D-torus, that read as
\begin{equation}\label{EqNS}
\begin{cases}
\partial_t \rho +\text{div}(\rho \vec u) =0,\\
\partial_t(\rho \vec u) + \text{div}(\rho \vec u\otimes \vec u) +\nabla P =\mu \Delta\vec u - \rho \vec{g},\\
\text{div}\vec u=0,
\end{cases}
\end{equation}
where $t\geqslant 0, \vec x=(x_1,x_2) \in 2\pi L\mathbb{T} \times (-1,1)$. The unknowns $\rho := \rho(t,\vec x)$, $\vec u :=  \vec u(t, \vec x)$ and $P :=P(t,\vec x)$ denote the density, the velocity and the pressure of the fluid, respectively, while $\mu$ is the viscosity coefficient and $ \vec{g}:= g{\vec e}_2$,  $g > 0$ being the gravitational constant.  Let $\Sigma_\pm =2\pi L \mathbb{T} \times \{\pm 1\}$,  the Navier-slip boundary conditions  proposed by  Navier (see \cite{Nav23}) are given  on $\Sigma_\pm $ as follows
\begin{equation}\label{EqNavierBoundary}
\begin{split}
\vec u \cdot \vec n&=0 \quad\text{on }\Sigma_+\cup\Sigma_-, \\
(\mu(\nabla \vec u+\nabla \vec u^T)\cdot \vec n )_{\tau}&= \xi(\vec x)\vec u \quad\text{on }\Sigma_+\cup\Sigma_-.
\end{split}
\end{equation}
Here, $\vec n$ is the outward normal vector of the boundary, $(\mu(\nabla \vec u+\nabla \vec u^T)\cdot \vec n )_{\tau}$ is the tangential component of $\mu(\nabla \vec u+\nabla \vec u^T)\cdot \vec n$ and $\xi(\vec x)$ is a scalar function describing the slip effect on the boundary,  only  taking \textit{nonnegative} constant values $\xi_{\pm}$  on $\Sigma_\pm $, respectively.

Let $\rho_0$ and $P_0$ be two $C^1$-functions on $x_2$ such that $P_0'=-g\rho_0$ with $'=d/dx_2$. Then, the laminar flow $(\rho_0(x_2),\vec 0, P_0(x_2))$ is a steady-state solution of \eqref{EqNS}. Of interest of this paper is to study the nonlinear  instability of the above laminar flow to Eq. \eqref{EqNS}-\eqref{EqNavierBoundary} that satisfies
\begin{equation}\label{RhoAssume}
\rho_0 \in C^1([-1,1]),  \quad \rho_0'>0 \text{ on } [-1,1], \quad \rho_0(\pm 1)=\rho_{\pm} \in (0,+\infty),
\end{equation}
i.e. to study the nonlinear Rayleigh--Taylor instability problem.

Linearizing \eqref{EqNS} in the vicinity of  $(\rho_0(x_2),\vec 0, P_0(x_2))$  and then seeking a  normal mode at a horizontal spatial frequency $k \in L^{-1}\mathbb{Z}\setminus\{0\}$ of the form 
\[
e^{\lambda( k) t} \vec U(\vec x) =e^{\lambda(k) t} (\cos(kx_1)\omega(x_2), \sin(kx_1)\theta(x_2),\cos(kx_1)\phi(x_2),\cos(kx_1)q(x_2))^T,
\]
the linear RT instability amounts to the investigation of the parameter $\lambda( k)\in \mathbb{C}$ (Re$\lambda>0$) such that  there exists a nontrivial solution $\phi \in H^4((-1,1))$ of the following ordinary differential equation for the second component of velocity 
\begin{equation}\label{MainEq}
-\lambda^2 (\rho_0 k^2 \phi - (\rho_0 \phi')') = \lambda \mu (\phi^{(4)} - 2k^2 \phi'' + k^4 \phi) - gk^2 \rho_0'\phi,
\end{equation}
 with the boundary conditions 
\begin{equation}\label{NavierBound}
\begin{cases}
\phi(-1)=\phi(1)=0,\\
\mu \phi''(1) = \xi_+ \phi'(1), \\
\mu \phi''(-1) =- \xi_- \phi'(-1).
\end{cases}
\end{equation}
Note that the embedding $H^4((-1,1)) \hookrightarrow C^3((-1,1))$ allows us to write \eqref{NavierBound}.
In this case, such a $\lambda$  is called a growth rate of the instability or a characteristic value of the linearized problem (see Eq. \eqref{SystemModes} below) as in \cite[Sect. 92-93, Chap. X]{Cha61}). We will present the derivation of the physical model in Section \ref{SectMainResults}. 

As the density profile is increasing, we  first show that  $\lambda$ is always real in Lemma \ref{LemEigenvalueReal}. Since our goal is to study the instability, we are left to look for $\lambda>0$.  Hence, for the linear instability, we continue  the spectral analysis of Helffer and Lafitte \cite{HL03}, Lafitte and Nguyễn \cite{LN20} for Eq. \eqref{MainEq}-\eqref{NavierBound}.  

For any horizontal spatial frequency $k \in L^{-1}\mathbb{Z}\setminus \{0\}$, we then define a $k$-supercritical regime of the viscosity coefficient $\mu >\mu_c(k,\Xi)$ (see $\mu_c(k,\Xi)$ in Proposition \ref{PropMuC}), we  prove that  there exists an infinite sequence of  characteristic values $(\lambda_n(k,\mu))_{n\geq 1}$, decreasing towards 0 as $n\to \infty$. This is stated in Theorem \ref{MainThm}.

The second goal, described in Section \ref{SectProofNonLinear} is to obtain a nonlinear instability result on more general initial data using the linear result of Theorem \ref{MainThm} (see \eqref{GeneralInitialCond}) and working in the regime $\mu>3\sup_{k\in  L^{-1}\mathbb{Z}\setminus \{0\}}\mu_c(k,\Xi)$.  The classical way of proving the nonlinear instability is to estimate the difference between the solution to the nonlinear problem and the normal mode solution to the linearized problem.  In  order to show that, in the spirit of Guo-Strauss \cite{GS95} and Grenier \cite{Gre00}, only the maximal normal mode  $ e^{\lambda_1 t} \vec U_1(\vec x)$ was taken  to derive a solution of the nonlinear  equation whose initial datum is $\delta \vec U_1(\vec x)$ with $0<\delta\ll 1$.  Our nonlinear result, Theorem \ref{ThmNonlinear}, generalizes the previous results of Guo-Strauss and of Grenier, by showing that a wide class of initial data (related to a linear combination of normal modes) to the nonlinear problem departing from the equilibrium gives rise to the nonlinear instability. 

This paper is organized as follows. In Section \ref{SectMainResults}, we present the governing equations and state the main results. Section \ref{Preliminaries} is devoted to some materials for the linear study. Then, in Section \ref{SectProofThmLinear}, we prove the linear instability, i.e. Theorem \ref{MainThm}.  Section \ref{SectProofNonLinear} is to prove the nonlinear instability, i.e. Theorem \ref{ThmNonlinear}. 

\section{Main results}\label{SectMainResults}

\subsection{The governing equations}

Let us recall the steady state $(\rho_0(x_2), \vec 0, P_0(x_2))$ of \eqref{EqNS}, with $\rho_0$ satisfies \eqref{RhoAssume} and $P_0'=-g\rho_0$.
We now derive the linearization of Eq. \eqref{EqNS} around the equilibrium state  $(\rho_0(x_2), \vec 0,P_0(x_2))$.  The perturbations 
\[\sigma =\rho-\rho_0,\quad \vec u = \vec u- \vec 0, \quad p = P-P_0\]
 thus satisfy
\begin{equation}\label{EqPertur}
\begin{cases}
\partial_t \sigma + \vec u \cdot \nabla(\rho_0+\sigma)  =0,\\
(\rho_0+\sigma) \partial_t \vec u + (\rho_0+\sigma) \vec u \cdot \nabla \vec u +\nabla p =\mu \Delta \vec u - \sigma \vec g,\\
\text{div} \vec v=0.
\end{cases}
\end{equation}
Note that  $(\mu(\nabla \vec u+\nabla \vec u^T)\cdot \vec n )_{\tau} = \vec n \times (\mu(\nabla \vec u+\nabla \vec u^T)\cdot \vec n) \times \vec n$ and that $\vec n = (0,\pm 1)^T$. Hence, the boundary conditions are 
\begin{equation}\label{BoundLinearized}
\begin{cases}
u_2=0,&\quad\text{on } \Sigma_\pm, \\
\mu\partial_{x_2} u_1= \xi_+ u_1 &\quad\text{on } \Sigma_+,\\
\mu\partial_{x_2} u_1=-\xi_- u_1 &\quad\text{on } \Sigma_-.
\end{cases}
\end{equation}
The linearized equations are
\begin{equation}\label{EqLinearized}
\begin{cases}
\partial_t \sigma +  \rho_0' u_2=0,\\
\rho_0 \partial_t \vec u + \nabla p = \mu \Delta \vec u - \sigma \vec g,\\
\text{div}\vec u=0,
\end{cases}
\end{equation}
and the corresponding  boundary conditions remaining \eqref{BoundLinearized}.

The linear RT instability problem is   to seek a normal mode of the form 
\begin{equation}\label{NormalModes}
\begin{cases}
\sigma(t,\vec x) =  e^{\lambda t} \cos(kx_1)\omega(x_2),\\
u_1(t,\vec x) =  e^{\lambda t}\sin(kx_1)\theta(x_2),\\
u_2(t,\vec x)=  e^{\lambda t}\cos(kx_1)\phi(x_2),\\
q(t,\vec x) =  e^{\lambda t}\cos(kx_1)q(x_2).
\end{cases}
\end{equation}
where $k\in L^{-1}\mathbb{Z}\setminus \{0\}$, $\lambda \in \C \setminus \{0\}$ and $\text{Re} \lambda \geqslant 0$. 
It follows from \eqref{EqLinearized} that
\begin{equation}\label{SystemModes}
\begin{cases}
\lambda \omega +\rho_0'\phi=0,\\
\lambda \rho_0 \theta - k q + \mu (k^2 \theta- \theta'')=0,\\
\lambda \rho_0 \phi + q' +  \mu (k^2 \phi-\phi'')= -g\omega,\\
 k \theta+ \phi'=0
\end{cases}
\end{equation}
and from \eqref{BoundLinearized} that
\begin{equation}\label{BoundEigenfunctions}
\phi(\pm 1)=0, \quad \mu \theta'(1) = \xi_+ \theta(1), \quad \mu \theta'(-1)= -\xi_- \theta(-1).
\end{equation}
We  obtain 
\begin{equation}\label{EqPressure}
\begin{split}
\omega &= -\frac{\rho_0'}{\lambda}\phi, \quad \theta =-\frac1k \phi', \quad q = -\frac1{k^2} (\lambda\rho_0\phi' + \mu(k^2\phi'-\phi''')).
\end{split}
\end{equation}
Then, we substitute $q, \omega$ into $\eqref{SystemModes}_3$ to get a fourth-order ordinary differential equation \eqref{MainEq}. We have the boundary conditions \eqref{NavierBound} deduced from \eqref{BoundLinearized}, which are obtained by assuming the solution to be in $C^2([-1,1])$.

\subsection{Main results}

Before stating our main results, we present our material for the linearized equations. 

When the density profile $\rho_0$ is increasing, we  show that all characteristic values $\lambda$ are real. Let $L_0$ be the characteristic length, such that $L_0^{-1} = \|\frac{\rho_0'}{\rho_0}\|_{L^\infty((-1,1))}$, we further obtain the uniform upper bound $\sqrt{\frac{g}{L_0}}$ of $\lambda$.
\begin{lemma}\label{LemEigenvalueReal}
For any $ k \in L^{-1}\mathbb{Z}\setminus \{0\}$, 
\begin{itemize}
\item all  characteristic values $\lambda$ are always real,
\item all  characteristic values $\lambda$ satisfy that $\lambda \leqslant \sqrt{\frac{g}{L_0}}$.
\end{itemize}
\end{lemma}

Proof of Lemma \ref{LemEigenvalueReal} is given in Section \ref{AppRealEigenvalue_Navier}. In view of Lemma \ref{LemEigenvalueReal}, we look for functions $\phi$ being real and we only consider the vector space of real functions in what follows.

We now study the linearized problem, i.e. \eqref{MainEq}-\eqref{NavierBound}.  Of importance  is to construct a bilinear coercive form $\bB_{k,\lambda,\mu}$ as $\lambda\geq 0 $ and $k\in \R\setminus \{0\}$ (i.e. we do not restrict $\lambda \in (0,\sqrt\frac{g}{L_0})$ and $k\in L^{-1}\mathbb{Z}\setminus \{0\}$ at this step) on the functional space 
\[
\tilde H^2((-1,1)) :=\{\phi \in H^2((-1,1)), \phi(\pm 1)=0\},
\]
so that we can transform our problem into solving the variational problem 
\begin{equation}
\lambda \bB_{k,\lambda,\mu}(\phi, \theta) = gk^2\int_{-1}^1 \rho_0'\phi \theta dx_2 \quad\text{for all } \theta \in \tilde H^2((-1,1)),
\end{equation}
for all $\phi, \theta$ staying in the functional space $\tilde H^2((-1,1))$ associated with the norm $\sqrt{\bB_{k,\lambda,\mu}(\cdot,\cdot)}$.  The desired bilinear form $\bB_{k,\lambda,\mu}$ is 
\begin{equation}\label{BilinearForm}
\begin{split}
\bB_{k,\lambda,\mu}(\vartheta, \varrho) :=&\lambda \int_{-1}^1  \rho_0 (k^2\vartheta  \varrho + \vartheta'  \varrho') dx_2 +  \mu \int_{-1}^1 (\vartheta''  \varrho'' + 2k^2 \vartheta'  \varrho' +k^4 \vartheta  \varrho)dx_2 \\
&\qquad\quad- \xi_-\vartheta'(-1)\varrho'(-1) -\xi_+ \vartheta'(1)\varrho'(1),
\end{split}
\end{equation}
For all $\lambda \geq 0$ and $k\in \R\setminus \{0\}$, we will place ourselves in a $k$-supercritical regime of the viscosity coefficient  $\mu >\mu_c(k,\Xi)$ with $\Xi=(\xi_-,\xi_+)$ (see  the precise formula $\mu_c(k,\Xi) $ in Proposition \ref{PropMuC})  such that 
\begin{equation}\label{DefMuC_k}
\bB_{k,0,\mu}\text{ is coercive if and only if  } \mu >\mu_c(k,\Xi),
\end{equation}
it yields $\bB_{k,\lambda,\mu}$ is coercive for all $\lambda \geq 0$ and $\mu>\mu_c(k,\Xi)$.  In view of Riesz's representation theorem, we thus obtain an abstract  operator $Y_{k,\lambda,\mu}$ from $\tilde H^2((-1,1))$ to its dual, such that 
\begin{equation}
 \bB_{k,\lambda,\mu}(\phi, \theta) = \langle Y_{k,\lambda,\mu}\phi, \theta\rangle 
\end{equation}
for all $\theta \in \tilde H^2((-1,1))$.
It turns out  that  the  existence of $H^4((-1,1))$ solutions of Eq. \eqref{MainEq}-\eqref{NavierBound} on $(-1,1)$ is reduced to  the existence of  weak solutions $\phi \in \tilde H^2((-1,1)$ of
\[
\lambda Y_{k,\lambda,\mu}\phi = gk^2\rho_0' \phi
\]
on $(-1,1)$.   Owing to a bootstrap argument to solution $\phi$, we obtain (see Proposition \ref{PropInverseOfR_Navier}) $Y_{k,\lambda,\mu}$ explicitly, that is a fourth-order differential operator, 
\[
Y_{k,\lambda,\mu}\phi = \lambda( \rho_0 k^2 \phi - (\rho_0 \phi')')+ \mu (\phi^{(4)} - 2k^2 \phi'' + k^4 \phi) 
\]
and we get back that $\phi \in H^4((-1,1))$ satisfies \eqref{MainEq} on $(-1,1)$ and the boundary conditions \eqref{NavierBound}.

Denoting by $\cM$ the operator of multiplication by $\sqrt{\rho_0'}$ in $L^2((-1,1))$, we now find $(\lambda, v)$ such that 
\[
\frac{\lambda}{gk^2} v = \cM Y_{k,\lambda,\mu}^{-1}\cM v.
\] 
The theory of self-adjoint and compact operators for a Sturm-Liouville problem on the functional space $H^2((-1,1))$ plays a key role here. Once it is proven that the operator $\cM Y_{k,\lambda,\mu}^{-1}\cM$ is compact and self-adjoint from $L^2((-1,1))$ to itself, then its discrete spectrum is a sequence of  eigenvalues (denoted by $\gamma_n(k,\lambda,\mu)$). Let $v_{n,k,\lambda,\mu}$ be an eigenfunction  of $\cM Y_{k,\lambda,\mu}^{-1}\cM$ associated with the eigenvalue $\gamma_n(k,\lambda,\mu)$ and let $\phi_{n,k,\lambda,\mu} = Y_{k,\lambda,\mu}^{-1}\cM v_{n,k,\lambda,\mu}$, we have 
\[
\gamma_n(k,\lambda,\mu) Y_{k,\lambda,\mu} \phi_{n,k,\lambda,\mu} = \cM^2\phi_{n,k,\lambda,\mu}=\rho_0'\phi_{n,k,\lambda,\mu}.
\]
For each $n$, we solve the equation
\begin{equation}\label{EqFindLambda}
\gamma_n(k,\lambda,\mu)= \frac{\lambda}{gk^2}.
\end{equation}
We will show that Eq. \eqref{EqFindLambda} has a unique root $\lambda_n(k,\mu) \in \R_+$ because of the decrease  of $\gamma_n$  in $\lambda$, which is an easy extension of  Kato's perturbation theory of spectrum of operators \cite{Kato}.  In addition, when $\lambda_n$ is a characteristic value, we have $\lambda_n \leq \sqrt\frac{g}{L_0}$ for all $n\geq 1$. This yields that for any horizontal spatial frequency $ k \in L^{-1}\mathbb{Z}\setminus \{0\}$,  there exists a sequence of characteristic values $(\lambda_n(k,\mu))_{n\geq 1}$, that is uniformly bounded and we further obtain that $\lambda_n$ decreases towards 0 as $n\to \infty$. For each $\lambda_n$,  we have that $\phi_{n,k,\lambda_n,\mu} = Y_{k,\lambda_n,\mu}^{-1}v_{n,k,\lambda_n,\mu}$ is a solution in $H^4((-1,1))$ of \eqref{MainEq}-\eqref{NavierBound} associated with $\lambda=\lambda_n$.

We sum up the above arguments in our first theorem. 
\begin{theorem}\label{MainThm}
Let $k\in L^{-1}\mathbb{Z}\setminus \{0\}$ be fixed and let  $\rho_0$ satisfy that   \eqref{RhoAssume}, i.e. 
\[\rho_0 \in C^1([-1,1]), \quad \rho_0(\pm 1)=\rho_{\pm} \in (0,\infty), \quad \rho_0'>0 \text{ everywhere on } [-1,1].\]
For all $\mu>\mu_c(k,\Xi)$, there exists an infinite sequence  $(\lambda_n, \phi_n)_{n\geqslant 1}$ with $\lambda_n >0$ decreasing towards 0 and $\phi_n \in H^4((-1,1))$, $\phi_n$ non trivial,  satisfying \eqref{MainEq}-\eqref{NavierBound}. 
\end{theorem}

Once Eq. \eqref{MainEq}-\eqref{NavierBound} is solved, we go back to the linearized equations \eqref{EqLinearized}.  For a fixed $k\in L^{-1}\mathbb{Z}\setminus \{0\}$,  we obtain a sequence of solutions to the linearized equations \eqref{EqLinearized} as indicated in Proposition \ref{PropSolEqLinear}, which are $(e^{\lambda_j(k,\mu) t} \vec U_j(k,\vec x))_{j\geq 1}$, with $\vec U_j(k,\vec x)=(\sigma_j, \vec u_j, p_j)^T(\vec x)$.

Let us choose a $k_0\in L^{-1}\mathbb{Z}\setminus \{0\}$.  In view of getting  infinitely  many characteristic values of the linearized problem, we introduce  a linear combination of  unstable normal modes 
\begin{equation}\label{GeneralInitialCond}
\vec U^M(t, \vec x)= \sum_{j=1}^M \Csf_j e^{\lambda_j(k_0,\mu) t}\vec U_j(k_0,\vec x) 
\end{equation}
to construct  an approximate solution to the nonlinear problem \eqref{EqNS}-\eqref{EqNavierBoundary}, with   constants $\Csf_j$ being chosen such that 
\begin{equation}\label{NormalizedCond_Navier_1}
\text {at least one of }\Csf_j\text{ } (1\leq j\leq N) \text{ is non-zero}
\end{equation}
and let $j_m:= \min \{j: 1\leq j\leq N, \Csf_j \neq 0\}$,
\begin{equation}\label{NormalizedCond_Navier_2} 
\begin{split}
|\Csf_{j_m}| \| u_{j_m}\|_{L^2(\Omega)}> \frac12 \sum_{j\geq j_m+1}|\Csf_j|\| u_j\|_{L^2(\Omega)}.
\end{split}
\end{equation}
Assume Eq. \eqref{EqPertur}--\eqref{BoundLinearized} is supplemented with initial datum  $\delta U^M(0, x)$ ($0<\delta \ll 1$), there is a unique local strong solution $(\sigma^{\delta},  u^{\delta})$ with an associated pressure $q^{\delta}$ to the nonlinear  equations Eq. \eqref{EqPertur}--\eqref{BoundLinearized} on $[0, T_{\max})$ (see Proposition \ref{PropLocalSolution}). We define the differences 
\[
(\sigma^d, \vec u^d, q^d)= (\sigma^{\delta}, \vec u^{\delta}, q^{\delta})- \delta (\sigma^M,\vec u^M,q^M)
\]
and  estimate the bound of $\|(\sigma^d,\vec u^d)\|_{L^2(\Omega)}$  in time (see Proposition \ref{PropL2_NormU^d})  in the regime 
\begin{equation}\label{CriticalViscosity}
\mu >3\mu_c(\Xi), \quad \text{with } \mu_c(\Xi) := \sup_{k\in L^{-1}\mathbb{Z}\setminus\{0\}}\mu_c(k,\Xi).
\end{equation}
Indeed, since $\mu>3\mu_c(\Xi)$, we  can choose a constant $\varpi_0>0$ such that 
\begin{equation}\label{Nu_0}
\mu>(3+\varpi_0)\mu_c(\Xi).
\end{equation}
 Hence, $\nu_0 = \frac{3+\varpi_0}{2+\varpi_0} \in (1,\frac32)$. It follows from Theorem \ref{MainThm} and Lemma \ref{LemEigenvalueReal}(2) that  exists \begin{equation}\label{DefLambda}
0<\Lambda=\sup_{k \in L^{-1}\mathbb{Z}\setminus \{0\}}\lambda_1(k,\mu) \leq \sqrt\frac{g}{L_0}.
\end{equation}
We further look for   $k_0 \in L^{-1}\mathbb{Z}\setminus \{0\}$ to have that 
\begin{equation}\label{AssumeLambdaN}
\Lambda \geq \lambda_1(k_0,\mu) > \lambda_2(k_0,\mu)> \dots > \lambda_N(k_0,\mu) >\frac{2\nu_0}3 \Lambda >\lambda_{N+1}(k_0,\mu) > \dots.
\end{equation}
For $t$ small enough, we deduce the bound in time of $\|(\sigma^d,\vec u^d)(t)\|_{L^2(\Omega)}$ in Proposition \ref{PropL2_NormU^d}, that is 
\[\|(\sigma^d,\vec u^d)(t)\|_{L^2(\Omega)}^2 \leq C \delta^3 \Big(\sum_{j=j_m}^N |\Csf_j| e^{\lambda_j t}+ \max(0,M-N) \Big(\max_{N+1\leq j\leq M}|\Csf_j|\Big) e^{\frac23 \nu_0\Lambda t}\Big)^3.\]
The nonlinear result  follows. 
 
\begin{theorem}\label{ThmNonlinear}
Let $\mu_c(\Xi)$ be defined as in \eqref{CriticalViscosity} and $\mu>3\mu_c(\Xi)$. Let  $\rho_0$ satisfies \eqref{RhoAssume}, i.e. 
\[
\rho_0 \in C^1([-1,1]), \quad \rho_0(\pm 1)=\rho_{\pm} \in (0,\infty), \quad \rho_0'>0 \text{ everywhere on } [-1,1].
\]
Let $M\in \mathbb{N}^{\star}$, there exist a positive constant $m_0$ and two positive constants $\delta_0$ and  $ \epsilon_0 $ sufficiently small such that for any $\delta \in (0,\delta_0)$,  the nonlinear  equations \eqref{EqPertur} with boundary conditions \eqref{BoundLinearized} and the initial data 
\[
\delta \sum_{j=1}^M \Csf_j \vec U_j(\vec x)
\]
 satisfying \eqref{NormalizedCond_Navier_1}-\eqref{NormalizedCond_Navier_2} has a unique local strong solution  $(\sigma^{\delta}, \vec u^{\delta})$ with an associated pressure $q^{\delta}$  such that 
\begin{equation}\label{BoundU_2Tdelta}
\| \vec u^{\delta}(T^{\delta})\|_{L^2(\Omega)} \geq m_0\epsilon_0,
\end{equation}
where $T^{\delta}\in (0, T_{\max})$  is given by $\delta \sum_{j=j_m}^M|\Csf_j| e^{\lambda_j T^\delta }=\epsilon_0$.
\end{theorem}

\section{Preliminaries}\label{Preliminaries}

The first aim  is to prove Lemma \ref{LemEigenvalueReal_Navier}, showing that all characteristic values $\lambda$ are real for any increasing density profile $\rho_0$. In the second part, we find the exact formula of the $k$-critical viscosity coefficient $\mu_c(k,\Xi)$ (see \eqref{DefMuC_k} above) for all $k>0$. The last goal is to study the bilinear form $\bB_{k,\lambda,\mu}$ in Section \ref{SectBilinearForm} to prepare for our linear study.

\subsection{The positivity of characteristic values $\lambda$}\label{AppRealEigenvalue_Navier}
\begin{lemma}\label{LemEigenvalueReal_Navier}
For any $ k \in L^{-1}\mathbb{Z}\setminus \{0\}$, 
\begin{itemize}
\item all  characteristic values $\lambda$ are always real,
\item all  characteristic values $\lambda$ satisfy that $\lambda \leqslant \sqrt{\frac{g}{L_0}}$.
\end{itemize}
\end{lemma}
\begin{proof}
Let $\overline\phi\in H^4((-1,1))$ satisfy \eqref{MainEq}-\eqref{NavierBound}.  Multiplying by $ \phi$ on both sides of \eqref{MainEq} and  using the integration by parts, we get that
\[
\begin{split}
- \int_{-1}^1  (\rho_0 \phi')'\overline\phi dx_2 &= -\rho_0\phi' \overline\phi \Big|_{-1}^1 + \int_{-1}^1 \rho_0 |\phi'|^2 dx_2 
\end{split}
\]
that 
\[
-\int_{-1}^1 \phi'' \overline\phi dx_2 = -\phi' \overline\phi\Big|_{-1}^1 + \int_{-1}^1|\phi'|^2 dx_2
\]
and that
\[
\int_{-1}^1 \phi^{(4)} \overline\phi dx_2 = \phi''' \overline\phi \Big|_{-1}^1 - \phi'' \overline\phi' \Big|_{-1}^1  + \int_{-1}^1 |\phi''|^2 dx_2,
\]
we  obtain that 
\begin{equation}\label{EqVariational_Navier}
\begin{split}
&\lambda \Big( \mu \int_{-1}^1 ( |\phi''|^2 + 2k^2 |\phi'|^2 + k^4 |\phi|^2) dx_2 - \xi_-|\phi'(-1)|^2 -\xi_+|\phi'(1)|^2\Big)\\
&\qquad+\lambda^2 \int_{-1}^1 ( k^2 \rho_0 |\phi|^2 + \rho_0 |\phi'|^2 ) dx_2 = gk^2 \int_{-1}^1 \rho_0'|\phi|^2 dx_2. 
\end{split}
\end{equation}
Suppose that $\lambda = \lambda_1 + i\lambda_2$, then one deduces from \eqref{EqVariational_Navier} that 
\begin{equation}\label{EqRealPart_Navier}
\begin{split}
& \lambda_1 \Big( \mu \int_{-1}^1 ( |\phi''|^2 + 2k^2 |\phi'|^2 + k^4 |\phi|^2 ) dx_2- \xi_-|\phi'(-1)|^2 -\xi_+|\phi'(1)|^2\Big)\\
&\qquad+(\lambda_1^2-\lambda_2^2) \int_{-1}^1 ( k^2 \rho_0 |\phi|^2 + \rho_0 |\phi'|^2 ) dx_2  = gk^2 \int_{-1}^1 \rho_0'|\phi|^2 dx_2
\end{split}
\end{equation}
and that 
\begin{equation}\label{EqImaginaryPart_Navier}
\begin{split}
& \lambda_2 \Big(\mu \int_{-1}^1 ( |\phi''|^2 + 2k^2 |\phi'|^2 + k^4 |\phi|^2) dx_2- \xi_-|\phi'(-1)|^2 -\xi_+|\phi'(1)|^2\Big) \\
&=-2\lambda_1 \lambda_2 \int_{-1}^1 ( k^2 \rho_0 |\phi|^2 + \rho_0 |\phi'|^2 ) dx_2.\end{split}
\end{equation}
If $\lambda_2 \neq 0$, Eq. \eqref{EqImaginaryPart_Navier} leads us to
\[
\begin{split}
-2\lambda_1 \int_{-1}^1( k^2 \rho_0 |\phi|^2 + \rho_0 |\phi'|^2 ) dx_2 &=  \mu \int_{-1}^1 ( |\phi''|^2 + 2k^2 |\phi'|^2 + k^4 |\phi|^2 )dx_2\\
&\qquad - \xi_- |\phi'(-1)|^2 -\xi_+ |\phi'(1)|^2,
\end{split}
\]
which yields 
\[
\begin{split}
-(\lambda_1^2-\lambda_2^2) \int_{-1}^1 ( k^2 \rho_0 |\phi|^2 + \rho_0 |\phi'|^2 ) dx_2  &= -2\lambda_1^2 \int_{-1}^1 ( k^2 \rho_0 |\phi|^2 + \rho_0 |\phi'|^2 ) dx_2 \\
&\qquad - gk^2 \int_{-1}^1 \rho_0'|\phi|^2 dx_2.
\end{split}
\]
Equivalently, 
\begin{equation}\label{EqModuloLambda_Navier}
(\lambda_1^2+\lambda_2^2)\int_{-1}^1 ( k^2 \rho_0 |\phi|^2 + \rho_0 |\phi'|^2 ) dx_2= - gk^2 \int_{-1}^1 \rho_0' |\phi|^2 dx_2.
\end{equation}
That implies 
\[
(\lambda_1^2+\lambda_2^2) k^2 \rho_- \int_{-1}^1 |\phi|^2 dx_2 \leqslant - gk^2 \int_{-1}^1 \rho_0'|\phi|^2 dx_2.
\]
The positivity of $\rho_0'$ yields a contradiction, then  $\lambda$ is real. Due to \eqref{EqRealPart_Navier} again, we further get that 
\[
\lambda^2 \int_{-1}^1\rho_0(k^2|\phi|^2+|\phi'|^2) dx_2 \leqslant gk^2 \int_{-1}^1\rho_0'|\phi|^2 dx_2.
\]
It tells us that $\lambda$ is  bounded by  $\sqrt{\frac{g}{L_0}}$. This finishes the proof of Lemma \ref{LemEigenvalueReal_Navier}.
\end{proof}

Note again that, thanks to Lemma \ref{LemEigenvalueReal}, in what follows in this section, we only use real-valued functions for the linear analysis. 

\subsection{The threshold of viscosity coefficient}\label{SectThreshold}

We obtain the precise formula of the critical viscosity coefficient $\mu_c(k,\Xi)$  for all $k\in \R\setminus\{0\}$. Note that  $\mu_c(k,\Xi)=\mu_c(-k,\Xi)$ for all $k\in \R\setminus \{0\}$, it suffices to find $\mu_c(k,\Xi)$ for $k\in \R_+$.
\begin{proposition}\label{PropMuC}
The following results hold.
\begin{enumerate}
\item For all $k\in \R_+$, we have 
\begin{equation}\label{MuC_kMax}
\mu_c(k,\Xi) = \max\limits_{\phi\in \tilde H^2((-1,1))} \frac{\xi_-(\phi'(-1))^2+\xi_+(\phi'(1))^2}{\int_{-1}^1 ((\phi'')^2+2k^2(\phi')^2+k^4\phi^2) dx_2}.
\end{equation}
Moreover,
\begin{equation}\label{EqMuK}
\mu_c(k,\Xi)= \frac1{4k \sinh^2(2k) } 
\left( \begin{split} 
&( \sinh(2k)\cosh(2k) -2k)(\xi_+ +\xi_-) \\
&+ \left(\begin{split} 
&( \sinh(2k) -2k\cosh(2k))^2 (\xi_++ \xi_-)^2  \\
& + \sinh^2(2k)(\sinh^2(2k)-4k^2) (\xi_+ -\xi_-)^2
\end{split}\right)^{\frac12}
\end{split}\right).
\end{equation}

\item $\mu_c(k,\Xi)$ is a decreasing function in $k\in \R_+$ and 
\begin{equation}\label{Mu_Climit}
\lim_{k\to 0} \mu_c(k,\Xi) = \sup_{k\in \R\setminus\{0\}}\mu_c(k,\Xi) = : \mu_c^s(\Xi).
\end{equation}
 We have the asymptotic expansion of $\mu_c(k,\Xi)$ as $k\to 0^+$, 
\begin{equation}\label{LimitMuK}
\begin{split}
\mu_c(k,\Xi) &=\frac13 \Big(\xi_++ \xi_-+ \sqrt{\xi_+^2-\xi_+\xi_- +\xi_-^2}\Big)\\
 &\qquad -\frac2{15} \Big( 4(\xi_++\xi_-) + \frac{4\xi_+^2-\xi_+\xi_- +4\xi_-^2}{\sqrt{\xi_+^2-\xi_+\xi_-+\xi_-^2}} \Big) k^2 + O(k^3).
 \end{split}
\end{equation}
That implies 
\begin{equation}\label{EqMuC-MuC_k}
\mu_c^s(\Xi)  = \frac13 \Big(\xi_++ \xi_-+ \sqrt{\xi_+^2-\xi_+\xi_- +\xi_-^2}\Big).
\end{equation}
As $k \gg 1$,  we obtain the limit
\begin{equation}\label{UpperBoundMuK}
\mu_c(k,\Xi)  \leq \frac{\sqrt{2(\xi_+^2+\xi_-^2)}}k \to 0.
\end{equation}

\item We have
 \begin{equation}\label{EqMuC}
\begin{split}
\mu_c^s(\Xi) &=\max\limits_{\phi\in \tilde H^2((-1,1))} \frac{\xi_-(\phi'(-1))^2+\xi_+(\phi'(1))^2}{\int_{-1}^1 (\phi'')^2 dx_2}\\
&=  \frac13 \Big(\xi_++ \xi_-+ \sqrt{\xi_+^2-\xi_+\xi_- +\xi_-^2}\Big).
\end{split}
\end{equation}

\end{enumerate}
\end{proposition}
 The proof is postponed to Appendix \ref{FormulaMuK}.

\begin{remark}
We see that  $\mu_c^s(\Xi)$ in \cite[Proposition 2.2]{DLX18}
should be revised and we redo the computation the critical viscosity coefficient. The authors in \cite{DLX18} consider $2\pi L\bT\times (0,1)$ instead of $2\pi L\bT\times (-1,1)$ and constant values $k_{0,1}$ instead of $\xi_{\pm}$. The formula of the critical viscosity defined as in \cite[(1.29)]{DLX18} is 
\[
\mu_c := \sup_{\phi \in \mathcal{Y}} Z(\phi),
\]
where 
\[
\begin{split}
\mathcal{Y} &=\{\phi \in H_0^1((0,1))\cap H^2((0,1)), \frac12 \int_0^1|\phi''|^2dx=1\},\\
Z(\phi) &= \frac{k_0}2|\phi'(0)|^2+\frac{k_1}2|\phi'(1)|^2.
\end{split}
\]
The authors in \cite[Proposition 2.2]{DLX18} claim that $\mu_c= \frac{k}6$ if $k_0=k_1=k>0$. However, in that case, let us take $\phi = \frac1{\sqrt2}(x^2-x) \in \mathcal{Y}$, then  we have a contradiction that
\[
Z\Big( \frac1{\sqrt2}(x^2-x)\Big) = \frac{k}2 > \frac{k}6.
\]
\end{remark}

\subsection{A bilinear form and a self-adjoint invertible operator}\label{SectBilinearForm}

In what follows in this section we have $\lambda\geq 0$ and  $k\in \R_+$ being fixed. Let us recall the definition of $\bB_{k,\lambda,\mu}$ from \eqref{BilinearForm},
\[ \begin{split}
\bB_{k,\lambda,\mu}(\vartheta, \varrho) & := \lambda \int_{-1}^1  \rho_0 (k^2\vartheta  \varrho + \vartheta'  \varrho') dx_2 +  \mu \int_{-1}^1 (\vartheta''  \varrho'' + 2k^2 \vartheta'  \varrho' +k^4 \vartheta  \varrho)dx_2 \\
&\qquad\quad- \xi_-\vartheta'(-1)\varrho'(-1) -\xi_+ \vartheta'(1)\varrho'(1).
\end{split} \]

\begin{lemma}\label{LemPropertyR_Navier}
We have the followings.
\begin{itemize}
\item For all $ \mu >0$, $\bB_{k,\lambda,\mu}$ is a continuous bilinear form on $\tilde H^2((-1,1))$. 
\item  For all  $\mu>\mu_c(k,\Xi)$, the bilinear form $\bB_{k,\lambda,\mu}$ is coercive. 
\end{itemize}
\end{lemma}
\begin{proof}[Proof of Lemma \ref{LemPropertyR_Navier}]
Clearly, $\bB_{k,\lambda,\mu}$ is a bilinear form on  $\tilde H^2((-1,1)) $. 
We then establish the boundedness of $\bB_{k,\lambda,\mu}$.  The integral terms of $\bB_{k,\lambda,\mu}$ are bounded by 
\begin{equation}\label{Bound34termB}
C^\star(\lambda+1) \|\vartheta\|_{\tilde H^2((-1,1))}\|\varrho\|_{\tilde H^2((-1,1))},
\end{equation}
where $C^\star$ is generic constant depending on physical parameters.  Meanwhile, it follows from the general Sobolev inequality that  
\[
(\vartheta'(-1))^2+ (\vartheta'(1))^2 \leq C^\star  \|\vartheta'\|_{H^1((-1,1))}^2.
\]
 Consequently, we get 
\begin{equation}\label{BoundBcontinuous_Navier}
|\bB_{k,\lambda,\mu}(\vartheta, \varrho)| \leq C^\star (1+\lambda )  \|\vartheta\|_{\tilde H^2((-1,1))} \|\varrho\|_{\tilde H^2((-1,1))},
\end{equation}
i.e. $\bB_{k,\lambda,\mu}$ is bounded.

We  show the coercivity of $\bB_{k,\lambda,\mu}$. We have that 
\[
\begin{split}
\bB_{k,\lambda,\mu}(\vartheta,\vartheta) &= \lambda\int_{-1}^1\rho_0(k^2 \vartheta^2+(\vartheta')^2)dx_2 + \mu \int_{-1}^1( (\vartheta'')^2+ 2k^2 (\vartheta')^2 +k^4 \vartheta^2)dx_2 \\
&\qquad -\xi_-(\vartheta'(-1))^2 -\xi_+ (\vartheta'(1))^2.
\end{split}
\]
As $\lambda\geq 0$ and $\mu >\mu_c(k,\Xi)$,  we have
\begin{equation}\label{LowerBoundB_2_Navier}
\begin{split}
\bB_{k,\lambda,\mu}(\vartheta,\vartheta) &\geq \lambda \int_{-1}^1\rho_0(k^2\vartheta^2+(\vartheta')^2)dx_2 \\
&\qquad\quad+ (\mu-\mu_c(k,\Xi)) \int_{-1}^1((\vartheta'')^2+ 2k^2 (\vartheta')^2 +k^4\vartheta^2)dx_2\\
&\geq (\mu-\mu_c(k,\Xi)) \int_{-1}^1( (\vartheta'')^2+ 2k^2 (\vartheta')^2 +k^4\vartheta^2)dx_2.
\end{split}
\end{equation}
It then  follows from \eqref{BoundBcontinuous_Navier} and  \eqref{LowerBoundB_2_Navier} that $\bB_{k,\lambda,\mu}$ is a  continuous and coercive bilinear form on $\tilde H^2((-1,1))$.
\end{proof}

With the above property of $\bB_{k,\lambda,\mu}$, we then establish:
\begin{proposition}\label{PropPropertyY_Navier}
Let  $\mu>\mu_c(k,\Xi)$ and    $(\tilde H^2((-1,1)))'$ be the dual space of $\tilde H^2((-1,1))$, associated with the norm $\sqrt{\bB_{k,\lambda,\mu}(\cdot,\cdot)}$. There is a unique operator 
\[
Y_{k,\lambda,\mu}\in \mathcal{L}(H^2((-1,1)), (\tilde H^2((-1,1)))'),
\]
which is also bijective, such that
\begin{equation}\label{EqMathcalB}
\bB_{k,\lambda,\mu}(\vartheta, \varrho) = \langle Y_{k,\lambda,\mu}\vartheta,  \varrho\rangle
\end{equation}
for all $\vartheta, \varrho \in \tilde H^2((-1,1))$.
\end{proposition}
\begin{proof}
It follows from Riesz's representation theorem that there exists an operator $Y_{k,\lambda,\mu} \in \mathcal{L}(\tilde H^2((-1,1)), (\tilde H^2((-1,1)))')$ such that 
\[
\bB_{k,\lambda,\mu}(\vartheta,\varrho) =\langle Y_{k,\lambda,\mu}\vartheta, \varrho \rangle
\]
for all $\varrho \in \tilde H^2((-1,1))$. Proof of Proposition \ref{PropPropertyY_Navier} is complete.
\end{proof}

\begin{proposition}\label{PropInverseOfR_Navier}
We have the following results.
\begin{enumerate}
\item For all $\vartheta \in \tilde H^2((-1,1))$, 
\[
Y_{k,\lambda,\mu}\vartheta=\lambda(k^2\rho_0\vartheta -(\rho_0\vartheta')')+ \mu(\vartheta^{(4)}-2k^2\vartheta''+k^4\vartheta)
\]
  in $ \mathcal{D}'((-1,1))$.  

\item Let $f\in L^2((-1,1))$ be given, there exists a unique solution  $\vartheta \in \tilde H^2((-1,1))$ of 
\begin{equation}
Y_{k,\lambda,\mu}\vartheta = f \text{ in } (\tilde H^2((-1,1)))'.
\end{equation}
Moreover, we have that  $\vartheta \in H^4((-1,1))$ satisfies  the boundary conditions \eqref{NavierBound}.
\end{enumerate}
\end{proposition}
\begin{proof}
It follows from Proposition \ref{PropPropertyY_Navier} that there is a unique $\vartheta \in \tilde H^2((-1,1))$ such that
\begin{equation}\label{EqIntegralBform_Navier}
\lambda \int_{-1}^1  \rho_0 (k^2\vartheta  \varrho + \vartheta'  \varrho') dx_2+ \mu \int_{-1}^1 (\vartheta''  \varrho'' + 2k^2 \vartheta'  \varrho' +k^4 \vartheta  \varrho)dx_2 = \langle Y_{k,\lambda,\mu}\vartheta,  \varrho \rangle
\end{equation}
for all $\varrho \in C_0^{\infty}((-1,1))$.
We respectively define $(\vartheta'')'$ and $(\vartheta'')''$ in the distributional sense as the first and second derivative of $\vartheta''$ which is in $L^2((-1,1))$. Hence, Eq. \eqref{EqIntegralBform_Navier} is equivalent to
\begin{equation}\label{EqIntegralBform2_Navier}
\lambda \int_{-1}^1  \rho_0 (k^2\vartheta  \varrho + \vartheta'  \varrho') dx_2 +\mu \langle (\vartheta'')'',  \varrho\rangle + \mu \int_{-1}^1( 2k^2 \vartheta'  \varrho' +k^4 \vartheta  \varrho)dx_2 = \langle Y_{k,\lambda,\mu}\vartheta,  \varrho \rangle
\end{equation}
for all $\varrho \in C_0^{\infty}((-1,1))$. 
We deduce from \eqref{EqIntegralBform2_Navier} that
\begin{equation}\label{EqIntegralTransformB_Navier}
\lambda \int_{-1}^1 (k^2 \rho_0 \vartheta -(\rho_0 \vartheta')')  \varrho dx_2 + \mu \langle (\vartheta'')'' -2k^2 \vartheta'' + k^4 \vartheta,  \varrho \rangle =\langle Y_{k,\lambda,\mu}\vartheta,  \varrho \rangle
\end{equation}
for all $\varrho \in C_0^{\infty}((-1,1))$. The resulting equation implies  that
\begin{equation}\label{EqThetaDspace_Navier}
\mu ((\vartheta'')'' - 2k^2 \vartheta''  +k^4 \vartheta) +\lambda (k^2\rho_0 \vartheta - (\rho_0 \vartheta')') = Y_{k,\lambda,\mu}\vartheta  \quad\text{in } \mathcal{D}'((-1,1)). 
\end{equation}
The first assertion holds.

Under the assumption $f\in L^2((-1,1))$, we  improve the regularity of the weak solution $\vartheta \in \tilde H^2((-1,1))$ of \eqref{EqThetaDspace_Navier}. Indeed,  we rewrite \eqref{EqThetaDspace_Navier} as 
\[
\mu \langle (\vartheta'')'' , \varrho \rangle=  \int_{-1}^1 (Y_{k,\lambda,\mu}\vartheta +2\mu k^2 \vartheta'' -\mu k^4 \vartheta - \lambda k^2 \rho_0 \vartheta + \lambda (\rho_0 \vartheta')')  \varrho dx_2
\]
for all $\varrho \in C_0^{\infty}((-1,1))$. Since $(f+2\mu k^2 \vartheta'' -\mu k^4 \vartheta - \lambda k^2 \rho_0 \vartheta + \lambda (\rho_0 \vartheta')')$ belongs to $L^2((-1,1))$, it then follows from \eqref{EqIntegralTransformB_Navier} that   $(\vartheta'')'' \in L^2((-1,1))$.  Let $\chi \in C_0^\infty((-1,1))$ satisfy $\int_{-1}^1\chi(y)dy=1$. Using the distribution theory, we define $\Sigma \in \mathcal{D}'((-1,1))$ such that 
\begin{equation}\label{EqDefinePsi_Navier}
\langle \Sigma, \theta \rangle = \langle (\vartheta'')'', \zeta_\theta \rangle
\end{equation}
for all $\theta \in C_0^{\infty}((-1,1))$, where 
\[
\zeta_\theta(x_2) = \int_{-1}^{x_2}\Big( \theta(y) - \chi(y) \int_{-1}^1 \theta(s) ds\Big) dy
\]
for all $-1<x_2<1$. We  obtain 
\[
\langle \Sigma', \theta \rangle = - \langle \Sigma, \theta' \rangle = -  \langle (\vartheta'')'', \zeta_{\theta'} \rangle.
\]
Note that 
\[
\langle(\vartheta'')'', \zeta_{\theta'} \rangle = \langle (\vartheta'')'', \theta(x_2) - \int_{-1}^{x_2} \chi(y) \int_{-1}^1 \theta'(s)ds dy \rangle = \langle (\vartheta'')'', \theta \rangle, 
\]
this yields $\langle \Sigma', \theta \rangle = -\langle (\vartheta'')'', \theta\rangle$. Hence, we have  that $(\vartheta'')'+ \Sigma \equiv \text{constant}$. In view of $(\vartheta'')'' \in L^2((-1,1))$ and \eqref{EqDefinePsi_Navier}, we know that $(\vartheta'')' \in L^2((-1,1))$. Since $\vartheta \in \tilde H^2((-1,1))$ and $(\vartheta'')', (\vartheta'')'' \in L^2((-1,1))$, it tells us that $\vartheta$ belongs to $H^4((-1,1))$ and we can take their traces of derivatives of $\vartheta$ up to order 3.

By  performing \eqref{EqIntegralTransformB_Navier}, we then show that $\vartheta$ satisfies \eqref{NavierBound}. Indeed, for all $\varrho \in \tilde H^2((-1,1))$, we perform  the integration by parts to obtain from \eqref{EqIntegralTransformB_Navier} that
\[
\begin{split}
 &\lambda \int_{-1}^1  \rho_0 (k^2\vartheta  \varrho + \vartheta'  \varrho') dx_2 +  \mu \int_{-1}^1 (\vartheta''  \varrho'' + 2k^2 \vartheta'  \varrho' +k^4 \vartheta  \varrho)dx_2 \\
&\quad- \lambda \rho_0\vartheta' \varrho \Big|_{-1}^1  + \mu \Big(  \vartheta''' \varrho \Big|_{-1}^1 - \vartheta'' \varrho' \Big|_{-1}^1 - 2k^2  \vartheta' \varrho \Big|_{-1}^1 \Big) = \int_{-1}^1 (Y_{k,\lambda,\mu}\vartheta ){\varrho} dx_2.
\end{split}
\]
It then follows from the definition of the bilinear form $\bB_{k,\lambda,\mu}$ that
\begin{equation}\label{EqBvImply_Navier}
 \lambda \rho_0\vartheta' \varrho \Big|_{-1}^1  - \mu \Big(  \vartheta''' \varrho \Big|_{-1}^1 - \vartheta'' \varrho' \Big|_{-1}^1 - 2k^2  \vartheta' \varrho \Big|_{-1}^1 \Big)=\xi_-\vartheta'(-1)\varrho'(-1) +\xi_+ \vartheta'(1)\varrho'(1),
\end{equation}
for all  $\varrho \in \tilde H^2((-1,1))$.  By collecting all terms  corresponding to $\varrho'(\pm 1)$ in  \eqref{EqBvImply_Navier}, we deduce that
\[
\mu \vartheta''(\pm 1)=\pm \xi_{\pm}\vartheta'(\pm 1).
\]
This yields that $\vartheta$ satisfies \eqref{NavierBound}.
The proof of Proposition \ref{PropInverseOfR_Navier} is complete.
\end{proof}

We obtain more information on the inverse operator $Y_{k,\lambda,\mu}^{-1}$.
\begin{proposition}\label{PropInverseT_Navier}
The operator $Y_{k,\lambda,\mu}^{-1} : L^2((-1,1)) \to L^2((-1,1))$ is compact and self-adjoint. 
\end{proposition}
\begin{proof}
It follows from Proposition \ref{PropInverseOfR_Navier} that $Y_{k,\lambda,\mu}$, being supplemented with \eqref{NavierBound}, admits an inverse operator $Y_{k,\lambda,\mu}^{-1}$ from $L^2((-1,1))$ to a subspace of $H^4((-1,1))$ requiring all elements satisfy \eqref{NavierBound}, which is  symmetric  due to Proposition \ref{PropPropertyY_Navier}. We compose $Y_{k,\lambda,\mu}^{-1}$ with the continuous injection from $H^4((-1,1))$ to $L^2((-1,1))$. Notice that  the embedding $H^p((-1,1)) \hookrightarrow H^q((-1,1))$ for $p>q\geqslant 0$ is compact. Therefore,   $Y_{k,\lambda,\mu}^{-1}$ is compact and self-adjoint from $L^2((-1,1))$ to $L^2((-1,1))$. 
\end{proof}

\section{Linear instability}\label{SectProofThmLinear}

\subsection{A sequence of characteristic values}

We continue considering $\lambda \geq 0$ and  $k\in L^{-1}\mathbb{Z}\setminus\{0\}$ being fixed. We study the operator $S_{k,\lambda,\mu} := \cM Y_{k,\lambda,\mu}^{-1}\cM$, where $\cM$ is the operator of multiplication by $\sqrt{\rho'_0}$. 
\begin{proposition}
Under the hypothesis \eqref{RhoAssume}, the operator $S_{k,\lambda,
\mu} : L^2((-1,1)) \to L^2((-1,1))$ is compact and self-adjoint.
\end{proposition}
\begin{proof}
Due to the assumption of $\rho_0$ \eqref{RhoAssume}, the operator $S_{k,\lambda,\mu}$ is well-defined and bounded from $L^2((-1,1))$ to itself. $Y_{k,\lambda,\mu}^{-1}$ is compact, so is $S_{k,\lambda,\mu}$.   Moreover, because both  the inverse $Y_{k,\lambda,\mu}^{-1}$ and $\cM$ are self-adjoint, the self-adjointness of $S_{k,\lambda,\mu}$ follows. 
\end{proof}
As a result of the spectral theory of compact and self-adjoint operators, the point spectrum of $S_{k,\lambda,\mu}$ is discrete, i.e. is a  sequence $\{\gamma_n(k,\lambda,\mu)\}_{n\geqslant 1}$ of   eigenvalues of $S_{k,\lambda,\mu}$ that tends to 0 as $n\to \infty$, associated with normalized orthogonal eigenvectors $\{\varpi_n\}_{n\geqslant 1}$ in $L^2((-1,1))$.  That means 
\[
S_{k,\lambda,\mu}\varpi_n= \cM Y_{k,\lambda,\mu}^{-1}\cM \varpi_n= \gamma_n(k,\lambda,\mu) \varpi_n .
\]
So that  $\phi_n = Y_{k,\lambda,\mu}^{-1}\cM \varpi_n$ belongs to  $H^4((-1,1))$ and satisfies \eqref{NavierBound}.  One  thus has
\begin{equation}\label{EqRf_n}
\gamma_n(k,\lambda,\mu) Y_{k,\lambda,\mu}\phi_n =  \rho_0' \phi_n
\end{equation}
and $\phi_n$ satisfies \eqref{NavierBound}. Eq. \eqref{EqRf_n} also tells us that $\gamma_n(k,\lambda,\mu) >0$ for all $n$. Indeed, we obtain 
\[
\gamma_n(k,\lambda,\mu)\int_{-1}^1 (Y_{k,\lambda,\mu}\phi_n)  \phi_n dx_2= \int_{-1}^1 \rho_0'\phi_n^2 dx_2.
\]
That implies
\begin{equation}\label{EqPsi_nB}
\gamma_n(k,\lambda,\mu) \bB_{k,\lambda,\mu}(\phi_n,\phi_n) = \int_{-1}^1 \rho_0'\phi_n^2 dx_2.
\end{equation}
Since $\bB_{k,\lambda,\mu}(\phi_n,\phi_n) >0$ and $\rho_0' >0$ on $(-1,1)$, we know that $\gamma_n(k,\lambda,\mu)$ is positive. Hence, by reordering, we have that $\gamma_n(k,\lambda,\mu)$ is a positive sequence decreasing towards 0 as $n\to \infty$.

For each $n$,   $\phi_n$ is a solution of \eqref{MainEq}-\eqref{NavierBound} if and only if  there are positive $\lambda_n$  such that \eqref{EqFindLambda} holds.
To solve \eqref{EqFindLambda}, we use the two following lemmas.
\begin{lemma}\label{AlphaCont}
For each $n$, 
\begin{itemize}
\item $\gamma_n(k,\lambda,\mu)$ and $\phi_n$ are differentiable in $\lambda$.
\item $\gamma_n(k,\lambda,\mu)$ is decreasing in $\lambda$.
\end{itemize}
\end{lemma}
\begin{proof}
The proof of Lemma \ref{AlphaCont}(1) is the same as \cite[Lemma 3.3]{LN20}, we omit the details here. We now prove that $\gamma_n(k,\lambda,\mu)$ is decreasing in $\lambda$.

Let $z_n= \frac{d\phi_n}{d\lambda}$, it follows from  \eqref{EqRf_n} that 
\begin{equation}\label{1stEqDeriTz_n}
k^2\rho_0\phi_n- (\rho_0\phi_n')'+ Y_{k,\lambda,\mu}z_n = \frac1{\gamma_n(k,\lambda,\mu)} \rho_0'z_n+ \frac{d}{d\lambda}\Big( \frac1{\gamma_n(k,\lambda,\mu)}\Big)\rho_0'\phi_n
\end{equation}
on $(-1,1)$. At $x_2=\pm 1$, we have 
\begin{equation}\label{BoundZ_n}
\begin{cases}
z_n(-1)=z_n(1)=0,\\
\mu z_n''(1)= \xi_+ z_n'(1),\\
\mu z_n''(-1)= -\xi_- z_n'(-1).
\end{cases}
\end{equation}
Multiplying by $\phi_n$ on both sides of \eqref{1stEqDeriTz_n}, we obtain that 
\begin{equation}\label{2ndEqDeriTz_n}
\begin{split}
&\int_{-1}^1 (k^2\rho_0\phi_n- (\rho_0\phi_n')') \phi_n dx_2+ \int_{-1}^1 (Y_{k,\lambda,\mu}z_n) \phi_n dx_2 \\
&\qquad\qquad= \frac1{\gamma_n(k,\lambda,\mu)} \int_{-1}^1 \rho_0'z_n \phi_n dx_2+ \frac{d}{d\lambda}\Big( \frac1{\gamma_n(k,\lambda,\mu)}\Big) \int_{-1}^1\rho_0'\phi_n^2 dx_2.
\end{split}
\end{equation}
Note that $z_n$ enjoys \eqref{BoundZ_n}, then 
\[
\int_{-1}^1 (Y_{k,\lambda,\mu}z_n) \phi_n dx_2 = \int_{-1}^1(Y_{k,\lambda,\mu}\phi_n)z_n dx_2 = \frac1{\gamma_n(k,\lambda,\mu)} \int_{-1}^1 \rho_0'z_n \phi_n dx_2.
\]
That implies
\begin{equation}\label{3rdEqDeriTz_n}
\frac{d}{d\lambda}\Big( \frac1{\gamma_n(k,\lambda,\mu)}\Big) \int_{-1}^1\rho_0'\phi_n^2 dx_2 = \int_{-1}^1(k^2\rho_0\phi_n- (\rho_0\phi_n')') \phi_n dx_2.
\end{equation}
Using the integration by parts, we obtain from \eqref{3rdEqDeriTz_n} that 
\[
\frac{d}{d\lambda}\Big( \frac1{\gamma_n(k,\lambda,\mu)}\Big) \int_{-1}^1\rho_0'\phi_n^2 dx_2= \int_{-1}^1 \rho_0(k^2\phi_n^2+(\phi_n')^2)dx_2 >0.
\]
Consequently, $\gamma_n(k,\lambda,\mu)$ is decreasing in $\lambda>0$.
\end{proof}

\subsection{Proof of Theorem \ref{MainThm} and unstable mode solutions of the linearized equations}

In view of Lemma \ref{AlphaCont}, we are able to prove Theorem \ref{MainThm}.

\begin{proof}[Proof of Theorem \ref{MainThm}]
For each $n$, there is only one solution $\lambda_n$  of \eqref{EqFindLambda}. Indeed, using \eqref{EqPsi_nB}, we know that
\[
\frac1{\gamma_n(k,\lambda,\mu)} \int_{-1}^1 \rho_0' \phi_n^2 dx_2= \int_{-1}^1 (Y_{k,\lambda,\mu}\phi_n) \phi_n dx_2 = \bB_{k,\lambda,\mu}(\phi_n,\phi_n).
\]
Hence, it follows from \eqref{LowerBoundB_2_Navier} that 
\[
\begin{split}
\frac1{\gamma_n(k,\lambda,\mu)} \int_{-1}^1 \rho_0' \phi_n^2 dx_2 &\geq \lambda\int_{-1}^1\rho_0(k^2\phi_n^2+(\phi_n')^2) dx_2 \\
&\qquad+(\mu-\mu_c(k,\Xi)) \int_{-1}^1((\phi_n'')^2+2k^2(\phi_n')^2+k^4\phi_n^2)dx_2 \\
&\geq \lambda k^2 \int_{-1}^1\rho_0\phi_n^2 dx_2 +(\mu-\mu_c(k,\Xi)) k^4 \int_{-1}^1\phi_n^2dx_2.
\end{split}
\]
That implies
\[
\frac1{L_0\gamma_n(k,\lambda,\mu)}  \geqslant \lambda  k^2+ \frac{(\mu-\mu_c(k,\Xi))k^4}{\rho_+}.
\]
Consequently, for all $n\geqslant 1$,
\begin{equation}\label{LimitGammaRight}
 \frac{\lambda}{\gamma_n(k,\lambda,\mu)} > gk^2 \text{ for } \lambda \text{ large}.
\end{equation}
Meanwhile, for all $n\geq 1$ and $\lambda\leq \frac12 \sqrt\frac{g}{L_0}$,
\begin{equation}\label{LimitGammaLeft}
 \frac{\lambda}{\gamma_n(k,\lambda,\mu)} \leq \frac{\lambda}{\gamma_n(k,\frac12\sqrt\frac{g}{L_0},\mu)} \to 0  \text{ as } \lambda\to 0.
\end{equation}
In view of \eqref{LimitGammaRight}, \eqref{LimitGammaLeft} and Lemma \ref{AlphaCont}, we obtain only one solution $\lambda_n$ of \eqref{EqFindLambda} and 
$(\lambda_n,\phi_n)$ satisfies \eqref{MainEq}-\eqref{NavierBound}. That means for all $n$, $\lambda_n$ is a characteristic value, hence it is bounded by $\sqrt\frac{g}{L_0}$.

We now prove that $(\lambda_n)_{n\geq1}$ decreases towards 0 as $n\to \infty$. If $\lambda_m<\lambda_{m+1}$ for some $m\geq 1$, we have  
\[
\gamma_m(k,\lambda_m,\mu) > \gamma_m(k,\lambda_{m+1},\mu).
\]
Meanwhile, we also have
\[
\gamma_m(k,\lambda_{m+1},\mu) > \gamma_{m+1}(k,\lambda_{m+1},\mu).
\]
That implies 
\[
\frac{\lambda_m}{gk^2}= \gamma_m(k,\lambda_m,\mu) >  \gamma_{m+1}(k,\lambda_{m+1},\mu) = \frac{\lambda_{m+1}}{gk^2}.
\]
That contradiction tells us that $(\lambda_n)_{n\geq 1}$ is a decreasing sequence.  Suppose that 
\[
\lim_{n\to \infty} \lambda_n = d_0 > 0.
\]
Note that for all $n$, $\gamma_n(k,\lambda_n,\mu) = \frac{\lambda_n}{gk^2}$, then
\[
\gamma_n(k,d_0,\mu) \geq \gamma_n(k,\lambda_n,\mu) = \frac{\lambda_n}{gk^2}.
\]
Let $n\to \infty$, we get that $0\geq d_0$, a contradiction. Hence $\lambda_n$ decreases towards 0 as $n\to \infty$. The proof of Theorem \ref{MainThm} is complete. 
\end{proof}

We now solve the linearized equations \eqref{EqLinearized} to prepare for our nonlinear part. 

\begin{proposition}\label{PropSolEqLinear}
For each $k\in L^{-1}\mathbb{Z}\setminus \{0\}\setminus\{0\}$ and for all $\mu>\mu_c(k,\Xi)$, there exists an infinite sequence of solutions $(n\geq 1)$
\[
\begin{split}
e^{\lambda_n(k,\mu) t} \vec U_n(k,\vec x) 
&=e^{\lambda_n(k,\mu) t}(\sigma_n, \vec u_n, p_n)^T(k,\vec x) \\
&= e^{\lambda_n(k,\mu) t} 
\begin{pmatrix}
\cos(kx_1)\omega_n(k,x_2) \\
 \sin(kx_1)\theta_n(k,x_2) \\
  \cos(kx_1)\phi_n(k,x_2) \\
   \cos(kx_1)q_n(k,x_2)
\end{pmatrix}
\end{split}
\]
to the linearized equation \eqref{EqLinearized}-\eqref{BoundLinearized}, such that 
\[
 \sigma_n \in H^2(\Omega), \vec u_n \in (H^3(\Omega))^2 \text{ and } p_n \in H^1(\Omega).
 \] 
\end{proposition}
\begin{proof}
For each   solution $\lambda_n\in  (0,\sqrt{\frac{g}{L_0}}) $ of \eqref{EqFindLambda}, we then have a solution $\phi_n= Y_{k,\lambda_n,\mu}^{-1}\cM \varpi_n \in H^4((-1,1))$ of \eqref{MainEq}-\eqref{NavierBound} as $\lambda=\lambda_n$.  We then find  a solution  to the system \eqref{SystemModes} as $\lambda=\lambda_n$. 
First, we obtain $\theta_n =-\frac{\phi_n'}k$ and $\omega_n= - \frac{\rho_0'\phi_n}{\lambda_n}$.
Due to \eqref{EqPressure},  we get
\[
q_n = -\frac1{k^2}(\lambda_n\rho_0 \phi_n' +\mu(k^2\phi_n'-\phi_n''')) \in H^1((-1,1)).
\]
With a solution $(\omega_n, \theta_n, \phi_n, q_n)$ of \eqref{SystemModes},  we then conclude that 
\[\begin{split}
&e^{\lambda_n(k,\mu) t}(\sigma_n, u_{n,1}, u_{n,2}, p_n)^T(k,\vec x)\\
&= e^{\lambda_n(k,\mu) t} (\cos(kx_1)\omega_n(k,x_2), \sin(kx_1)\theta_n(k,x_2), \cos(kx_1)\phi_n(k,x_2), \cos(kx_1)q_n(k,x_2))^T
\end{split}\]
is a solution to the linearized equations \eqref{EqLinearized}-\eqref{BoundLinearized}.
\end{proof}

\section{Nonlinear instability}\label{SectProofNonLinear}

\subsection{The local existence}
The first important things are the local existence of strong solutions to the nonlinear  equations and \textit{a priori} nonlinear energy estimates to those solutions. We restate Proposition 4.1 of \cite{DJL20}.

\begin{proposition}\label{PropLocalSolution}
Suppose that the steady state satisfies \eqref{RhoAssume}. For any given initial data $(\sigma_0, \vec u_0) \in (H^1(\Omega) \cap L^{\infty}(\Omega)) \times (H^2(\Omega))^2$ satisfying $\text{div}\vec u_0=0$, and also being compatible with the boundary conditions \eqref{EqNavierBoundary}, the nonlinear  equations \eqref{EqPertur} has a local strong solution
\begin{equation}\label{RegularityOfSolutions}
(\sigma,\vec u,\nabla q) \in C([0,T^{\max}), H^1(\Omega) \times (H^2(\Omega))^2 \times (L^2(\Omega)))^2.
\end{equation}
Let $\cE(t):=\sqrt{\|\sigma(t)\|_{H^1(\Omega)}^2 +\|\vec u(t)\|_{H^2(\Omega)}^2}$ and $\delta_0>0$ be sufficiently small, we further get that  if  $\sup_{0\leq s\leq t}\cE(t) \leq \delta_0$, there holds 
\begin{equation}\label{NonlinearEnergyIne}
\begin{split}
\cE^2(t) + \|(\nabla q,\partial_t\vec u)\|_{L^2(\Omega)}^2 &+ \int_0^t (\|\partial_t \vec u(s)\|_{H^1(\Omega)}^2+\|\vec u(s)\|_{H^2(\Omega)}^2) ds \\
&\leq C_0 \Big( \cE^2(0)+ \int_0^t \|(\sigma,\vec u)(s)\|_{L^2(\Omega)}^2 ds \Big).
\end{split}
\end{equation}
\end{proposition}

Thanks to Proposition \ref{PropSolEqLinear}, we will formulate a sequence of approximate solutions $e^{\lambda_n(k,\mu)}\vec U_n(k,\vec x)$ to the nonlinear  equations \eqref{EqPertur}-\eqref{BoundLinearized}, which are solutions to the linearized equations \eqref{EqLinearized}-\eqref{BoundLinearized}. Let us fix a $k=k_0$ such that \eqref{AssumeLambdaN} holds. For  $\delta>0$, we define
\begin{equation}
(\sigma^M,\vec u^M,q^M)(t,\vec x) :=   \sum_{j=1}^M e^{\lambda_j(k,\mu) t} \vec U_j(k,\vec x).
\end{equation}
Keeping in mind that $\min_{[-1,1]} \rho_0 >0$, then due to the embedding from $H^2(\Omega)$ to $L^{\infty}(\Omega)$, there exists a constant $\delta_0>0$ such that 
\begin{equation}\label{EqDelta_0}
\delta_0 \|\sum_{j\geq 1}\sigma_j(0,\vec x)\|_{L^{\infty}(\Omega)} > \frac12\min_{[-1,1]}\rho_0(x_2).
\end{equation}
Hence,
\[
\frac12 \min_{[-1,1]}\rho_0(x_2) \leq \min_{\Omega}(\rho_0(x_2)+\delta \sigma^M(0,\vec x))
\]
for $\delta \leq \delta_0$. By virtue of Proposition \ref{PropLocalSolution}, the nonlinear equations \eqref{EqPertur}-\eqref{BoundLinearized}  with the initial data $\delta(\sigma^M,\vec u^M,q^M)(0)$ admits a strong solution 
\[
(\sigma^{\delta},\vec u^{\delta}) \in C^0([0,T^{\max}),H^1(\Omega) \times (H^2(\Omega))^2)
\]
 with an associated pressure $q^{\delta} \in C^0([0,T^{\max}), L^2(\Omega))$. Furthermore, we have 
\[
\frac12 \min_{[-1,1]}\rho_0(x_2) \leq \inf_{\Omega}(\rho_0(x_2)+\sigma^\delta(t,\vec x))
\]
for all $t\in [0,T^{\max})$.
 
 In what follows, the constants $C_i(i\geq 1)$ are universal ones depending only on physical parameters, $M$ and $\Csf_j (j\geq 1)$. 

Let $F_M(t) = \sum_{j=j_m}^M|\Csf_j| e^{\lambda_j t}$ and $0<\epsilon_0 \ll 1$ be fixed later \eqref{EqEpsilon}. There exists  a unique $T^{\delta}$ such that $\delta F_M(T^{\delta})=\epsilon_0$. 
Let
\[
\begin{split}
C_1 &= \sqrt{\|\sigma^M(0)\|_{H^1(\Omega)}^2 + \|\vec u^M(0)\|_{H^2(\Omega)}^2},
\quad C_2 = \sqrt{\|\sigma^M(0)\|_{L^2(\Omega)}^2 + \|\vec u^M(0)\|_{L^2(\Omega)}^2}.
\end{split}
\]
We define
\begin{equation}\label{DefTstar}
\begin{split}
T^{\star} &:= \sup\Big\{t\in (0,T^{\max})| \cE(\sigma^{\delta}(t), \vec u^{\delta}(t)) \leq C_1\delta_0\}>0,\\
T^{\star\star} &:= \sup \{t\in (0,T^{\max})| \|(\sigma^{\delta},\vec u^{\delta})(t)\|_{L^2(\Omega)} \leq 2 C_2 \delta F_M(t)\Big\}>0.
\end{split}
\end{equation}
Note that $\cE(\sigma^{\delta}(0), \vec u^{\delta}(0)) = C_1\delta < C_1\delta_0$ and because of \eqref{RegularityOfSolutions}, we  then have $T^{\star}>0$. Similarly, we have $T^{\star\star}>0$. Then for all $t\leq \min\{T^{\delta}, T^{\star}, T^{\star\star}\}$, it follows from \eqref{NonlinearEnergyIne} that 
\begin{equation}\label{EnergyNormDelta}
\cE^2(\sigma^{\delta}(t), \vec u^{\delta}(t)) +\|\partial_t \vec u^{\delta}(t)\|_{L^2(\Omega)}^2+\int_0^t \|\nabla \partial_t\vec  u^{\delta}(\tau)\|_{L^2(\Omega)}^2d\tau \leq C_3 \delta^2 F_M^2(t).
\end{equation}

\subsection{The difference functions}

Let 
\[
\sigma^d= \sigma^{\delta} -\delta \sigma^M, \quad \vec u^d=\vec u^{\delta}- \delta \vec u^M, \quad q^d =q^{\delta}- \delta q^M.
\]
Then $(\sigma^d, \vec u^d, q^d)$ satisfies 
\begin{equation}\label{EqDiff}
\begin{cases}
\partial_t \sigma^d+ \rho_0' u_2^d = -\vec u^{\delta} \cdot \nabla \sigma^{\delta},\\
\rho_0 \partial_t \vec u^d-\mu \Delta \vec u^d+ \nabla q^d= -\sigma^{\delta}\partial_t\vec u^\delta- (\rho_0 +\sigma^{\delta}) \vec u^{\delta} \cdot \nabla \vec u^{\delta} - g\sigma^d \vec e_2,\\
\text{div} \vec u^d=0,
\end{cases}
\end{equation}
along with the initial condition,
\begin{equation}\label{InitialEqDiff}
(\sigma^d,\vec u^d)(0)=0
\end{equation}
and the boundary conditions,
\begin{equation}\label{BoundEqDiff}
\begin{cases}
u_2^d=0,&\quad\text{on } \Sigma_\pm, \\
\mu\partial_{x_2} u_1^d= \xi_+ u_1^d &\quad\text{on } \Sigma_+,\\
\mu\partial_{x_2} u_1^d=-\xi_- u_1^d  &\quad\text{on } \Sigma_-.
\end{cases}
\end{equation}
The compatibility conditions read as
\begin{equation}
u_1^d(0,x_1,-1)= u_1^d(0,x_1,1), \quad \text{div}\vec u^d(0)=0.
\end{equation}

We now establish the error estimate for $\|(\sigma^d,\vec u^d)\|_{L^2(\Omega)}$. 
\begin{proposition}\label{PropL2_NormU^d}
For all $t \leq \min(T^{\delta},T^{\star}, T^{\star\star})$, there holds
\begin{equation}\label{L2_NormU^d}
\|(\sigma^d,\vec u^d)(t)\|_{L^2(\Omega)}^2 \leq C_4 \delta^3 (\sum_{j=1}^N|\Csf_j| e^{\lambda_j t}+ \max(0,M-N)\max_{N+1\leq j\leq M}|\Csf_j| e^{\frac23 \nu_0\Lambda t})^3.
\end{equation}
\end{proposition}

The proof of Proposition \ref{PropL2_NormU^d} relies on Lemmas \ref{LemH^sNormU^mU^d}, \ref{LemNormU^dL^2At0} \ref{PropLambda_1}, \ref{LemIneMaximalMode}, \ref{LemEstMu_c} below.
\begin{lemma}\label{LemH^sNormU^mU^d}
We have the following inequalities
\begin{equation}\label{H^sNormU^mU^d}
\sum_{0\leq s\leq 2, 0\leq \tau\leq 1}\|\partial_t^{\tau} \vec u^d(t)\|_{H^s(\Omega)} \leq C_5 \delta F_M(t),
\end{equation}
and
\begin{equation}\label{H^sNormSigma^d}
\|\sigma^d(t)\|_{H^1(\Omega)}+ \|\partial_t \sigma^d(t)\|_{L^2(\Omega)} \leq C_6\delta F_M(t).
\end{equation}
\end{lemma}
\begin{proof}
For $\tau \in \{0 ,1\}$, 
\[
\partial_t^{\tau} \vec u^M(t) = \sum_{j=1}^M \lambda_j^{\tau} \Csf_j e^{\lambda_j t}\vec U_j(k_0,\vec x), 
\]
it yields,  for all $s\in \{0,1,2\}$,
\[
\|\partial_t^{\tau} \vec u^M(t)\|_{H^s(\Omega)} \leq C_7 F_M(t).
\]
In view of \eqref{EnergyNormDelta}, we then obtain that for $s\in \{0,1,2\}$ and $\tau \in \{0 ,1\}$,
\[
\|\partial_t^{\tau} \vec u^d(t)\|_{H^s(\Omega)} \leq \delta \|\partial_t^{\tau} \vec u^M(t)\|_{H^s(\Omega)}+ \|\partial_t^{\tau} \vec u^{\delta}(t)\|_{H^s(\Omega)} \leq C_8 \delta F_M(t).
\]

To prove \eqref{H^sNormSigma^d}, we use $\eqref{EqDiff}_1$ and \eqref{EnergyNormDelta} again, 
\[
\begin{split}
\|\sigma^d(t)\|_{H^1(\Omega)}+ \|\partial_t\sigma^d(t)\|_{L^2(\Omega)} &\leq \|\sigma^{\delta}(t)\|_{H^1(\Omega)}+\delta  \|\sigma^M(t)\|_{H^1(\Omega)}\\
&\qquad+ C_9 \|u_2^d(t)\|_{L^2(\Omega)}+ \|\vec u^{\delta}(t)\|_{L^2(\Omega)}\|\nabla \sigma^{\delta}\|_{L^2(\Omega)}\\
&\leq C_{10}\delta F_M(t).
\end{split}
\]
Lemma \ref{LemH^sNormU^mU^d} is proven.
\end{proof}

\begin{lemma}\label{LemNormU^dL^2At0}
There holds 
\begin{equation}\label{NormU^dL^2At0}
\|\partial_t\vec u^d(0)\|_{L^2(\Omega)}^2 \leq C_{11} \delta^3.
\end{equation}
\end{lemma}
\begin{proof}
From $\eqref{EqDiff}_{2,3}$ and the boundary conditions \eqref{BoundEqDiff}, we use the integration by parts to obtain that 
\[
\begin{split}
\int_\Omega \rho_0|\partial_t \vec u^d|^2 d\vec x &= \int_\Omega \mu \Delta\vec u^d \cdot  \partial_t \vec u^d d\vec x -\int_\Omega (\sigma^\delta \partial_t \vec u^\delta+(\rho_0+\sigma^\delta)\vec u^\delta\cdot \nabla\vec u^\delta)\cdot\partial_t \vec u^d d\vec x\\
&\qquad\quad - \int_\Omega g\sigma^d \partial_t u_2^d d\vec x.
\end{split}
\]
Thanks to Lemma \ref{LemH^sNormU^mU^d}, one has 
\begin{equation}
-\int_\Omega (\sigma^\delta \partial_t\vec u^\delta+(\rho_0+\sigma^\delta)\vec u^\delta\cdot \nabla\vec u^\delta)\cdot\partial_t\vec u^d d\vec x \leq C_{12} \delta^3 F_M^3(t).
\end{equation}
That implies 
\[
\|\partial_t\vec  u^d(t)\|_{L^2(\Omega)}^2 \leq C_{14} \Big( (\|\vec u^d(t)\|_{H^2(\Omega)} + \|\sigma^d(t)\|_{L^2(\Omega)})\|\partial_t\vec u^d(t)\|_{L^2(\Omega)} +  \delta^3F_M^3(t)\Big).
\]
Using Young's inequality, we further get 
\begin{equation}\label{NormU^dL^2At0_1}
\|\partial_t\vec u^d(t)\|_{L^2(\Omega)}^2 \leq C_{15}\Big( \|\vec u^d(t)\|_{H^2(\Omega)}^2 + \|\sigma^d(t)\|_{L^2(\Omega)}^2 + \delta^3F_M^3(t)\Big). 
\end{equation}
Letting $t\to 0$ in \eqref{NormU^dL^2At0_1}, we complete the proof Lemma \ref{LemNormU^dL^2At0}.
\end{proof}

We  derive the following lemma for $\lambda_1$.
\begin{lemma}\label{PropLambda_1}
The largest characteristic value $\lambda_1$  is the solution of the variational problem 
\begin{equation}
\frac1{gk^2} = \max_{\phi \in H^2((-1,1))} \frac{\int_{-1}^1 \rho_0' \phi^2dx_2}{\lambda_1\bB_{k,\lambda_1,\mu}(\phi,\phi)},
\end{equation}
where the bilinear form $\bB_{k,\lambda_1,\mu}$ is defined as in \eqref{BilinearForm} and $\phi_1$ is an extremal function. 
\end{lemma}
\begin{proof}
For all $\lambda>0$, we solve the variational problem 
\[
\beta(k,\lambda,\mu)= \max\Big( \int_{-1}^1 \rho_0' \phi^2dx_2 \Big| \phi \in \tilde H^2((-1,1)), \quad \lambda\bB_{k,\lambda,\mu}(\phi,\phi)=1\Big).
\]
Let us define the Lagrangian functional
\[
\cL_{\bB}(\phi, \beta) =  \int_{-1}^1 \rho_0' \phi^2dx_2 - \beta(\lambda\bB_{k,\lambda,\mu}(\phi,\phi)-1).
\]
Thanks to the Lagrange multiplier theorem, the extremal points of the quotient 
\[
\frac{\int_{-1}^1 \rho_0' \phi^2dx_2}{\lambda\bB_{k,\lambda,\mu}(\phi,\phi)}
\]
are necessarily the stationary points $( \beta_\star,\phi_\star)$ of $\cL_{\bB}$, which satisfy 
\begin{equation}\label{PsiConstraintAt1}
\lambda\bB_{k,\lambda,\mu}(\phi_\star,\phi_\star)=1
\end{equation}
and
\begin{equation}\label{EqConstraintLb}
\int_{-1}^1 \rho_0' \phi_\star \theta dx_2- \beta_\star\lambda\bB_{k,\lambda,\mu}(\phi_\star, \theta) =0,
\end{equation}
for all $\theta \in \tilde H^2((-1,1))$. Restricting $\theta \in C_0^{\infty}((-1,1))$ and following the line of the proof of Proposition \ref{PropInverseOfR_Navier}, one deduces from \eqref{EqConstraintLb} that $\phi_\star$ has to satisfy 
\begin{equation}
\beta_\star \lambda Y_{k,\lambda,\mu}\phi_\star = \rho_0'\phi_\star
\end{equation}
in a weak sense. We further get that $ \phi_\star\in H^4((-1,1))$ and satisfies \eqref{PsiConstraintAt1} and the boundary conditions \eqref{NavierBound}.  Hence,    $\lambda\beta_\star $ is an eigenvalue of the compact and self-adjoint operator $ S_{k,\lambda,\mu}$ from $L^2((-1,1))$ to itself, with $\cM^{-1}Y_{k,\lambda,\mu}\phi_\star \in L^2((-1,1))$ being an associated eigenfunction. That implies 
\begin{equation}\label{BetaLeqAlpha}
\beta(k,\lambda,\mu) \leq \lambda^{-1}\gamma_1(k,\lambda,\mu).
\end{equation}

Meanwhile, since the operator $S_{k,\lambda,\mu}$ is  self-adjoint and positive, we thus obtain that
\[
\gamma_1(k,\lambda,\mu) =\sup_{\omega \in L^2((-1,1))} \frac{\langle S_{k,\lambda,\mu}\omega, \omega\rangle}{\|\omega\|_{L^2((-1,1))}^2}.
\]
Hence, for all $\omega\in L^2((-1,1))$ and  for $\phi= Y_{k,\lambda,\mu}^{-1}\cM\omega \in H^4((-1,1))$, we have
\[
\gamma_1(k,\lambda,\mu) \langle Y_{k,\lambda,\mu}\phi,\phi \rangle \leq  \frac{\langle S_{k,\lambda,\mu}\omega, \omega\rangle^2}{\|\omega\|_{L^2((-1,1))}^2} \leq  \| S_{k,\lambda,\mu}\omega\|_{L^2((-1,1))}^2.
\]
Equivalently, 
\[
\gamma_1(k,\lambda,\mu) \leq \sup \Big\{ \frac{\|\cM \phi\|_{L^2((-1,1))}^2}{\langle Y_{k,\lambda,\mu}\phi,\phi \rangle}|\phi \in H^4((-1,1)) \text{ and } \cM^{-1}Y_{k,\lambda,\mu}\phi \in L^2((-1,1))\Big\},
\]
it yields 
\begin{equation}\label{AlphaLeqBeta}
\lambda^{-1}\gamma_1(k,\lambda,\mu) \leq \beta(k,\lambda,\mu).
\end{equation}

Two inequalities \eqref{BetaLeqAlpha} and \eqref{AlphaLeqBeta} tell us that $\beta(k,\lambda,\mu)=\lambda^{-1}\gamma_1(k,\lambda,\mu)$ for all $\lambda>0$,  then the proof of Lemma \ref{PropLambda_1} is complete.
\end{proof}
\begin{lemma}\label{LemIneMaximalMode}
Let
\[
\mathsf{X} := \{\vec w\in (H^2(\Omega))^2, \vec w \text{ satisfies }   \eqref{EqNavierBoundary} \text{ and div}\vec w=0\}.
\]
There holds for all $\vec w\in H_{\star}^2(\Omega)$,
\begin{equation}\label{IneMaximalMode}
\begin{split}
\int_{\Omega}g\rho_0'|w_2|^2 d\vec x &+ \Lambda  \int_{(2\pi L \bT)^2} (\xi_+|  w_1(x_1,1)|^2 + \xi_- |w_1(x_1,-1)|^2) dx_1 \\
&\leq \Lambda^2 \int_{\Omega}\rho_0|\vec w|^2 d\vec x+ \Lambda\mu \int_{\Omega}|\nabla \vec w|^2 d\vec x.
\end{split}
\end{equation}
\end{lemma}
The proof of Lemma \ref{LemIneMaximalMode}  is  due to the definition of $\Lambda$ \eqref{DefLambda} and Lemma \ref{PropLambda_1}, that is similar to \cite[Lemma 5.1]{DJL20}, hence we omit the details here.
\begin{lemma}\label{LemEstMu_c}
There holds for all $\vec w\in\mathsf{X}  \setminus \{\vec 0\}$,
\begin{equation}\label{EstMuC}
 \sup_{\vec w \in \mathsf{X}} \frac{\int_{2\pi L\bT} ( \xi_+ |w_1(x_1,1)|^2 +\xi_- |w_1(x_1,-1)|^2)  dx_1}{\|\nabla \vec w\|_{L^2(\Omega)}^2} \leq \mu_c(\Xi).
\end{equation}
\end{lemma}
\begin{proof}
Let us fix a horizontal frequency $k\in L^{-1}\mathbb{Z}$ and introduce the horizontal Fourier transform 
\[
\hat f(k,x_2)= \int_{2\pi L\bT}f(\vec x)e^{-i kx_1} dx_1.
\]
For $\vec w\in \mathsf{X}$, we write 
\[
\hat w_1(k,x_2)= -i\theta(k,x_2), \quad \hat w_2(k,x_2) = \phi(k,x_2).
\]
Then, $k\theta+\phi'=0$ and $(\theta,\phi)$ enjoy \eqref{BoundEigenfunctions}. Following Fubini's and Parseval's theorem, one thus deduces 
\begin{equation}\label{IntegralW_1Fourier}
\begin{split}
\int_{2\pi L\bT} ( \xi_+ |w_1(x_1,1)|^2 &+\xi_- |w_1(x_1,-1)|^2)  dx_1\\
& = \frac1{2\pi L}  \sum_{k\in L^{-1}\mathbb{Z}}  (\xi_+(|\theta(k,1)|^2+\xi_- |\theta(k,-1)|^2)
\end{split}
\end{equation}
and 
\begin{equation}\label{IntegralNablaW_Fourier}
\begin{split}
{\|\nabla \vec w\|_{L^2(\Omega)}^2}&= \frac1{2\pi L} \sum_{k\in L^{-1}\mathbb{Z}} \int_{-1}^1( k^2(|\theta|^2+|\phi|^2)+ |\theta'|^2+|\phi'|^2)(k,x_2) dx_2.
\end{split}
\end{equation}
We may reduce to estimate \eqref{EstMuC} when $\theta$ and $\phi$ are real-valued and  continue the estimate to the real and imaginary part of $\theta$ and $\phi$.
For any $k\in L^{-1}\mathbb{Z}\setminus\{0\}$, we  have from $k\theta+\phi'=0$ that 
\begin{equation}\label{EqThetaPsi'}
\xi_+(\theta(k,1))^2+\xi_- (\theta(k,-1))^2 = \frac1{k^2}(\xi_+((\phi'(k,1))^2+\xi_- (\phi'(k,-1))^2)
\end{equation}
and that
\begin{equation}\label{EqThetaPsiDeri}
\begin{split}
&\int_{-1}^1 \Big( k^2(\theta^2+\phi^2)+ (\theta')^2+(\phi')^2\Big)(k,x_2) 
dx_2 \\
&\qquad= \frac1{k^2} \int_{-1}^1 (k^4\phi^2+2k^2 (\phi')^2+(\phi'')^2)(k,x_2)dx_2.
\end{split}
\end{equation}
Owing to \eqref{IntegralW_1Fourier}, \eqref{EqThetaPsi'} and the definition of $\mu_c(k,\Xi)$ (see \eqref{MuC_kMax}), we get 
\[
\begin{split}
&\int_{2\pi L\bT} ( \xi_+ |w_1(x_1,1)|^2 +\xi_- |w_1(x_1,-1)|^2)  dx_1 \\
&\qquad\leq \frac1{2\pi L} \left(\begin{split} 
&\limsup_{k\to 0} \frac1{k^2} (\xi_+(\phi'(k,1))^2+\xi_- (\phi'(k,-1))^2)\\
&+ \sum_{k\in L^{-1}\mathbb{Z}\setminus\{0\}} \frac1{k^2} (\xi_+ (\phi'(k,1))^2+\xi_- (\phi'( k,-1))^2) \end{split}\right)\\
&\qquad\leq \frac1{2\pi L} \left(\begin{split} 
&\limsup_{k\to 0} \frac{\mu_c(k,\Xi)}{k^2}\int_{-1}^1 (k^4\phi^2+2k^2 \phi'^2+\phi''^2)( k,x_2)dx_2 \\
&+ \sum_{k\in L^{-1}\mathbb{Z}\setminus\{0\}} \frac{\mu_c(k,\Xi)}{k^2} \int_{-1}^1 (k^4\phi^2+2k^2 \phi'^2+\phi''^2)(k,x_2)dx_2
\end{split}\right).
\end{split}
\]
Thanks to Proposition \ref{PropMuC}, we obtain
\begin{equation}\label{1_LemEstMu_c}
\begin{split}
&\int_{2\pi L\bT} ( \xi_+ |w_1(x_1,1)|^2 +\xi_- |w_1(x_1,-1)|^2)  dx_1 \\
&\leq  \frac{\mu_c(\Xi)}{2\pi L} \left(\begin{split} 
&\limsup_{k\to 0} \frac1{k^2}\int_{-1}^1 (k^4 \phi^2+2k^2 (\phi')^2+(\phi'')^2)(k,x_2)dx_2 \\
&+ \sum_{k\in L^{-1}\mathbb{Z}\setminus\{0\}} \frac1{k^2} \int_{-1}^1 (k^4 \phi^2+2k^2 (\phi')^2+(\phi'')^2)(k,x_2)dx_2
\end{split}\right).
\end{split}
\end{equation}
Combining   \eqref{IntegralNablaW_Fourier}, \eqref{EqThetaPsiDeri} and \eqref{1_LemEstMu_c}, it gives
\[
\int_{2\pi L\bT} ( \xi_+ |w_1(x_1,1)|^2 +\xi_- |w_1(x_1,-1)|^2)  dx_1 \leq \mu_c(\Xi) \|\nabla \vec w\|_{L^2(\Omega)}^2.
\]
Lemma \ref{LemEstMu_c} is proven. 
\end{proof}

We now prove Proposition \ref{PropL2_NormU^d}.
\begin{proof}[Proof of Proposition \ref{PropL2_NormU^d}]
We rewrite $\eqref{EqDiff}_2$ as
\[
(\rho_0+\sigma^{\delta}) \partial_t \vec u^d-\mu \Delta \vec u^d + \nabla q^d= \vec f^\delta - g\sigma^d \vec e_2,
\]
where $\vec f^{\delta}= -\sigma^\delta \partial_t\vec u^M- (\rho_0 +\sigma^{\delta}) \vec u^{\delta} \cdot \nabla \vec u^{\delta}$. Differentiate the resulting equation with respect to $t$ and then multiply by $\partial_t\vec u^d$, we obtain after integration that 
\[\begin{split}
 &\int_{\Omega}\partial_t \sigma^{\delta}|\partial_t \vec u^d|^2 d\vec x+ \int_{\Omega} (\rho_0 +\sigma^{\delta}) \partial_t^2 \vec u^d \cdot \partial_t \vec u^d d\vec x \\
 &=  \int_{\Omega} \mu \Delta\partial_t \vec u^d \cdot \partial_t \vec u^d d\vec x- \int_{\Omega} \nabla \partial_t q^d \cdot \partial_t \vec u^d d\vec x +  \int_{\Omega}(\partial_t \vec  f^{\delta} -g\partial_t\sigma^d\vec e_2) \cdot\partial_t \vec u^d d\vec x.
\end{split}\]
Since $\text{div}\partial_t \vec u^d=0$, we use the integration by parts to further obtain 
\[\begin{split}
&\int_{\Omega}\partial_t\sigma^{\delta}(t)|\partial_t \vec u^d(t)|^2 d\vec x+ \int_{\Omega} (\rho_0 +\sigma^{\delta}(t)) \partial_t^2 \vec u^d(t) \cdot \partial_t \vec u^d(t) d\vec x \\
&=  \int_{\Omega}(\partial_t \vec f^{\delta}(t)-g\partial_t\sigma^d(t) \vec e_2)\cdot \partial_t \vec u^d(t)d\vec x - \mu\int_{\Omega}|\nabla \partial_t \vec u^d(t)|^2d\vec x \\
&\qquad\qquad+ \int_{2\pi L\bT} ( \xi_+ |\partial_t u_1^d(t,x_1,1)|^2 +\xi_- |\partial_t u_1^d(t, x_1,-1)|^2)  dx_1.
\end{split}\]
That means, 
\[\begin{split}
&\frac12\frac{d}{dt} \int_{\Omega}(\rho_0 +\sigma^{\delta}(t))|\partial_t \vec u^d(t)|^2 d\vec x \\
&=  -\frac12 \int_{\Omega}\sigma_t^{\delta}|\partial_t \vec u^d|^2 d\vec x + \int_{\Omega}(\partial_t \vec f^{\delta}(t)-g\partial_t \sigma^d(t) \vec e_2)\cdot \partial_t \vec u^d(t)d\vec x  - \mu\int_{\Omega}|\nabla \partial_t \vec u^d(t)|^2d\vec x \\
&\qquad+ \int_{2\pi L\bT} ( \xi_+ |\partial_t u_1^d(t,x_1,1)|^2 +\xi_- |\partial_t u_1^d(t, x_1,-1)|^2)  dx_1.
\end{split}\]
Using $\eqref{EqDiff}_1$, we then get 
\[
\begin{split}
&\frac{d}{dt} \int_{\Omega}\Big( (\rho_0+\sigma^{\delta}(t))|\partial_t \vec u^d(t)|^2-g\rho_0'|u_2^d(t)|^2\Big) d\vec x \\
&\qquad + 2  \mu\int_{\Omega}|\nabla \partial_t \vec u^d(t)|^2d\vec x  -2  \int_{2\pi L\bT} ( \xi_+ |\partial_t u_1^d(t,x_1,1)|^2 +\xi_- |\partial_t u_1^d(t, x_1,-1)|^2)  dx_1\\
&=-\int_{\Omega}\partial_t \sigma^{\delta}|\partial_t \vec u^d|^2 d\vec x+ 2\int_{\Omega}(\partial_t \vec f^{\delta}(t)+g \vec u^{\delta}(t) \cdot \nabla \sigma^{\delta}(t) \vec e_2)\cdot \partial_t \vec u^d(t) d\vec x. 
\end{split}
\]
Integrating in time variable, we get 
\begin{equation}\label{EstU_t^d_L^2}
\begin{split}
&\|\sqrt{\rho_0+\sigma^{\delta}(t)} \partial_t \vec u^d(t)\|_{L^2(\Omega)}^2 + 2 \mu\int_0^t \|\nabla \partial_t \vec u^d(s)\|_{L^2(\Omega)}^2  ds \\
&\qquad-2\int_0^t \int_{2\pi L\bT} ( \xi_+ |u_1^d(s,x_1,1)|^2 +\xi_- |u_1^d(s, x_1,-1)|^2)  dx_1ds\\
&= \int_{\Omega} g\rho_0' |u_2^d(t)|^2 d\vec x +\Big( \int_{\Omega}(\rho_0+\sigma^{\delta}(t))|\partial_t \vec u^d(t)|^2 d\vec x\Big)\Big|_{t=0}\\
&\qquad+ \int_0^t \int_{\Omega}(2\partial_t f^{\delta}(s)+2g \vec u^{\delta}(s)\cdot \nabla \sigma^{\delta}(s) \vec e_2 - \partial_t \sigma^{\delta}(s) \partial_t \vec u^d(s))\cdot \partial_t \vec u^d(s) ds.
\end{split}
\end{equation}
We continue using \eqref{H^sNormU^mU^d}, \eqref{H^sNormSigma^d} and \eqref{NormU^dL^2At0} to estimate each term of the r.h.s of \eqref{EstU_t^d_L^2}. This yields
\begin{equation}\label{2ndEstU_t^d_L^2}
\begin{split}
&\|\sqrt{\rho_0+\sigma^{\delta}(t)} \partial_t \vec u^d(t)\|_{L^2(\Omega)}^2 + 2 \mu\int_0^t \|\partial_t \vec u^d(s)\|_{L^2(\Omega)}^2  ds\\
&\qquad\qquad -2 \int_{2\pi L\bT} ( \xi_+ |\partial_t u_1^d(t,x_1,1)|^2 +\xi_- |\partial_t u_1^d(t, x_1,-1)|^2)  dx_1\\
&\leq \int_{\Omega} g\rho_0' |u_2^d(t)|^2 d\vec x + C_{16}\delta^3 F_M^3(t).
\end{split}
\end{equation}
Due to \eqref{IneMaximalMode}, we further get that 
\begin{equation}\label{2ndEstU_t^d_L^2}
\begin{split}
&\|\sqrt{\rho_0+\sigma^{\delta}(t)} \partial_t \vec u^d(t)\|_{L^2(\Omega)}^2 + 2 \mu\int_0^t \|\nabla \partial_t \vec u^d(s)\|_{L^2(\Omega)}^2  ds \\
&\qquad\qquad -2 \int_0^t \int_{2\pi L\bT} ( \xi_+ |\partial_t u_1^d(s,x_1,1)|^2 +\xi_- |\partial_t u_1^d(s, x_1,-1)|^2)  dx_1ds\\
&\leq \Lambda^2 \int_{\Omega} \rho_0 |\vec u^d(t)|^2 d\vec x+ \Lambda\mu \int_{\Omega}|\nabla \vec u^d(t)|^2 d\vec x \\
&\qquad\qquad -\Lambda \int_{2\pi L\bT} ( \xi_+ |u_1^d(t,x_1,1)|^2 +\xi_- |u_1^d(t, x_1,-1)|^2)  dx_1+ C_{16}\delta^3F_M^3(t) \\
&\leq  \Lambda^2 \int_{\Omega}(\rho_0+\sigma^{\delta}(t))|\vec u^d(t)|^2 d\vec x+ \Lambda\mu \int_{\Omega}|\nabla \vec u^d(t)|^2 d\vec x \\
&\qquad\qquad -\Lambda \int_{2\pi L\bT} ( \xi_+ |u_1^d(t,x_1,1)|^2 +\xi_- |u_1^d(t, x_1,-1)|^2)  dx_1+ C_{17}\delta^3F_M^3(t).
\end{split}
\end{equation}

On the other hand, we have
\[\begin{split}
&\frac{d}{dt}\|\sqrt{\rho_0 +\sigma^{\delta}(t)}\vec u^d(t)\|_{L^2(\Omega)}^2= 2\int_{\Omega} (\rho_0+\sigma^{\delta}(t)) \vec u^d(t) \cdot \partial_t \vec u^d(t) d\vec x + \int_{\Omega} \partial_t \sigma^{\delta}(t)|\vec u^d(t)|^2 d\vec x. 
\end{split}\]
Let us recall $\varpi_0$ from \eqref{Nu_0} and $ \nu_0 = \frac{3+\varpi_0}{2+\varpi_0} \in (1,\frac32)$. We fix two  positive constants $m_{1,2}$  such that
\begin{equation}\label{M1_constant}
m_1= \nu_0+\sqrt{\nu_0^2-1}
\end{equation}
and that 
\begin{equation}\label{M2_constant}
m_2 = \mu(m_1^2-m_1+1)-\mu_c(\Xi)(m_1^2+1) + \sqrt{(\mu(m_1^2-m_1+1)-\mu_c(\Xi)(m_1^2+1))^2-\mu^2m_1^2}.
\end{equation}
With $m_1>0$ from \eqref{M1_constant}, we use Young's inequality to observe
\[ \begin{split}
&2\int_{\Omega} (\rho_0+\sigma^{\delta}(t)) \vec u^d(t) \cdot \partial_t \vec u^d(t) d\vec x \\
&\leq \frac1{\Lambda m_1} \|\sqrt{\rho_0+\sigma^{\delta}(t)}\partial_t \vec u^d(t)\|_{L^2(\Omega)}^2 +\Lambda m_1\|\sqrt{\rho_0+\sigma^{\delta}(t)}\vec u^d(t)\|_{L^2(\Omega)}^2. 
\end{split}\]
That will imply 
\begin{equation}\label{IneDtU^d}
\begin{split}
\frac{d}{dt}\|\sqrt{\rho_0 +\sigma^{\delta}(t)}\vec u^d(t)\|_{L^2(\Omega)}^2 &\leq  \frac1{\Lambda m_1} \|\sqrt{\rho_0+\sigma^{\delta}(t)}\partial_t \vec u^d(t)\|_{L^2(\Omega)}^2 \\
&\qquad+\Lambda m_1\|\sqrt{\rho_0+\sigma^{\delta}(t)}\vec u^d(t)\|_{L^2(\Omega)}^2 + C_{18}\delta^3 F_M^3(t).
\end{split}
\end{equation}
With $m_2>0$ defined as in \eqref{M2_constant}, we  obtain from \eqref{2ndEstU_t^d_L^2} and \eqref{IneDtU^d} that
\[
\begin{split}
&\frac{d}{dt}\|\sqrt{\rho_0 +\sigma^{\delta}(t)}\vec u^d(t)\|_{L^2(\Omega)}^2 + m_2\|\nabla \vec u^d(t)\|_{L^2(\Omega)}^2 \\
&\leq \Big(m_1+ \frac1{m_1}\Big)\Lambda\|\sqrt{\rho_0 +\sigma^{\delta}(t)}\vec u^d(t)\|_{L^2(\Omega)}^2 + \Big( \frac{\mu}{m_1}+m_2\Big) \|\nabla\vec u^d\|_{L^2(\Omega)}^2\\
&\qquad +\frac2{\Lambda m_1} \int_0^t \int_{2\pi L\bT} ( \xi_+ |\partial_t u_1^d(s,x_1,1)|^2 +\xi_- |\partial_t u_1^d(s, x_1,-1)|^2)  dx_1ds\\
&\qquad - \frac{2\mu}{\Lambda m_1} \int_0^t\|\nabla\partial_t \vec u^d(s)\|_{L^2(\Omega)}^2 ds + C_{19} \delta^3F_M^3(t).
\end{split}
\]
Together with \eqref{EstMuC}, we deduce
\begin{equation}\label{IneEnergyForm1st}
\begin{split}
&\frac{d}{dt}\|\sqrt{\rho_0 +\sigma^{\delta}(t)}\vec u^d(t)\|_{L^2(\Omega)}^2 + m_2\|\nabla \vec u^d(t)\|_{L^2(\Omega)}^2 \\
&\leq \Big(m_1+ \frac1{m_1}\Big)\Lambda\|\sqrt{\rho_0 +\sigma^{\delta}(t)}\vec u^d(t)\|_{L^2(\Omega)}^2   + \Big( \frac{\mu}{m_1}+m_2\Big)  \|\nabla\vec u^d(t)\|_{L^2(\Omega)}^2\\
&\qquad\qquad- \frac{2(\mu-\mu_c(\Xi))}{\Lambda m_1} \int_0^t\|\nabla\partial_t \vec u^d(s)\|_{L^2(\Omega)}^2 ds + C_{19} \delta^3F_M^3(t).
\end{split}
\end{equation}
We use Young's inequality to get that 
\begin{equation}\label{IneNablaU^dYoung}
\begin{split}
\Big( \frac{\mu}{m_1}+m_2\Big)  \|\nabla\vec u^d(t)\|_{L^2(\Omega)}^2&= 2\Big( \frac{\mu}{m_1}+m_2\Big) \int_0^t \int_{\Omega}\nabla\vec u^d(s)\cdot \nabla\partial_t \vec u^d(s) d\vec xds\\
&\leq \frac{2(\mu-\mu_c(\Xi))}{\Lambda m_1} \int_0^t\|\nabla\partial_t \vec u^d(s)\|_{L^2(\Omega)}^2 ds  \\
&\qquad+\frac{\Lambda m_1 \Big( \frac{\mu}{m_1}+m_2\Big)^2}{2(\mu-\mu_c(\Xi))} \int_0^t\|\nabla\vec u^d(s)\|_{L^2(\Omega)}^2ds.
\end{split}
\end{equation}
Combining \eqref{IneEnergyForm1st} and \eqref{IneNablaU^dYoung} gives us 
\begin{equation}\label{Ine2ndBeforeGronwall}
\begin{split}
&\frac{d}{dt}\|\sqrt{\rho_0 +\sigma^{\delta}(t)}\vec u^d(t)\|_{L^2(\Omega)}^2 + m_2\|\nabla \vec u^d(t)\|_{L^2(\Omega)}^2 \\
&\leq  \Big(m_1+ \frac1{m_1}\Big)\Lambda\|\sqrt{\rho_0 +\sigma^{\delta}(t)}\vec u^d(t)\|_{L^2(\Omega)}^2  \\
&\qquad\quad+\frac{\Lambda m_1 \Big( \frac{\mu}{m_1}+m_2\Big)^2}{2(\mu-\mu_c(\Xi))} \int_0^t\|\nabla\vec u^d(s)\|_{L^2(\Omega)}^2ds +C_{19}\delta^3 F_M^3(t).
\end{split}
\end{equation}
It follows from \eqref{M1_constant} and \eqref{M2_constant} that
\[
\frac{\Lambda m_1\Big( \frac{\mu}{m_1}+m_2\Big)^2}{2(\mu-\mu_c(\Xi))}= \Lambda\Big(m_1+ \frac1{m_1}\Big) m_2=2\nu_0 \Lambda m_2.
\]
Therefore, \eqref{Ine2ndBeforeGronwall} becomes
\begin{equation}\label{IneBeforeGronwall}
\begin{split}
&\frac{d}{dt}\|\sqrt{\rho_0 +\sigma^{\delta}(t)}\vec u^d(t)\|_{L^2(\Omega)}^2 + m_2\|\nabla \vec u^d(t)\|_{L^2(\Omega)}^2 \\
&\leq  2\nu_0 \Lambda \Big(\|\sqrt{\rho_0 +\sigma^{\delta}(t)}\vec u^d(t)\|_{L^2(\Omega)}^2 + m_2 \int_0^t\|\nabla\vec u^d(s)\|_{L^2(\Omega)}^2ds \Big)+C_{19}\delta^3 F_M^3(t).
\end{split}
\end{equation}
Recalling that $\vec u^d(0)= \vec 0$, thus, applying Gronwall’s inequality to \eqref{IneBeforeGronwall}, one obtains
\begin{equation}\label{IneAfterGronwall}
\|\sqrt{\rho_0+\sigma^{\delta}(t)} \vec u^d(t)\|_{L^2(\Omega)}^2 + m_2 \int_0^t \|\nabla \vec u^d(s)\|_{L^2(\Omega)}^2 ds \leq C_{19}\delta^3 e^{2\nu_0 \Lambda t} \int_0^t e^{-2\nu_0\Lambda s}F_M^3(s) ds.
\end{equation}
Since $F_M^3(t)\leq M^2 \max_{j_m\leq j\leq M}|\Csf_j|^2F_M(3t)$, we then have  from \eqref{IneAfterGronwall} that
\begin{equation}\label{BoundUdFinal}
\begin{split}
\|\vec u^d(t)\|_{L^2(\Omega)}^2 &\leq C_{20} \delta^3 e^{2\nu_0 \Lambda t}  \sum_{j=j_m}^M \int_0^t |\Csf_j| e^{(3\lambda_j -2\nu_0\Lambda)s} ds.
\end{split}
\end{equation}
Because of \eqref{AssumeLambdaN}, we have $\lambda_j > \frac23 \nu_0\Lambda$ for $j_m\leq j\leq N$ and $\lambda_j <\frac23 \nu_0\Lambda$ for $j\geq N+1$. It yields that for $j_m \leq j\leq N$, 
\begin{equation}\label{IneIntLambda_j1N}
\int_0^t e^{(3\lambda_j-2\nu_0 \Lambda)s} ds = \frac1{3\lambda_j-2\nu_0\Lambda} (e^{(3\lambda_j-2\nu_0\Lambda)t}-1) \leq  \frac1{3\lambda_j-2\nu_0\Lambda} e^{(3\lambda_j-2\nu_0\Lambda)t}
\end{equation}
and that for $j\geq N+1$, 
\begin{equation}\label{IneIntLambda_jN+1}
\int_0^t e^{(3\lambda_j-2\nu_0 \Lambda)s} ds = \frac1{3\lambda_j-2\nu_0\Lambda} (e^{(3\lambda_j-2\nu_0\Lambda)t}-1)  \leq \frac1{2\nu_0\Lambda-3\lambda_j}.
\end{equation}
In view of \eqref{IneIntLambda_j1N} and \eqref{IneIntLambda_jN+1},  we obtain from \eqref{BoundUdFinal} that if $M\leq N$,
\[
\begin{split}
\|\vec u^d(t)\|_{L^2(\Omega)}^2 &\leq C_{20}  \delta^3 \Big( \sum_{j=j_m}^M \frac{|\Csf_j|}{3\lambda_j-2\nu_0\Lambda} e^{3\lambda_j t} \Big)
\end{split}
\]
 and if $M\geq N+1$, 
\[
\begin{split}
\|\vec u^d(t)\|_{L^2(\Omega)}^2 &\leq C_{20}  \delta^3 \Big( \sum_{j=j_m}^M \frac{|\Csf_j|}{3\lambda_j-2\nu_0\Lambda } e^{3\lambda_j t} + \sum_{j=N+1}^M  \frac{|\Csf_j|}{2\nu_0\Lambda-3\lambda_j} e^{2\nu_0\Lambda t} \Big).
\end{split}
\]
That means 
\begin{equation}\label{EstU^dGeq}
\|\vec u^d(t)\|_{L^2(\Omega)}^2 \leq C_{21} \delta^3 \Big(\sum_{j=j_m}^N |\Csf_j| e^{\lambda_j t}+ \max(0,M-N) \Big(\max_{N+1\leq j\leq M}|\Csf_j|\Big) e^{\frac23 \nu_0\Lambda t}\Big)^3.
\end{equation}

To show the bound of $\|\sigma^d(t)\|_{L^2(\Omega)}$, we use Cauchy-Schwarz's inequality to deduce from $\eqref{EqDiff}_1$ that 
\[
\frac{d}{dt}\|\sigma^d(t)\|_{L^2(\Omega)} \leq \|\sigma^d(t)\|_{L^2(\Omega)} \leq (\max \rho_0' ) \|u_2^d(t)\|_{L^2(\Omega)} + \|\vec u^\delta(t)\|_{L^2(\Omega)} \|\sigma^\delta(t)\|_{H^1(\Omega)}. 
\]
Using \eqref{EnergyNormDelta}, we obtain 
\[
\frac{d}{dt}\|\sigma^d(t)\|_{L^2(\Omega)} \leq  C_{22}(\|u_2^d(t)\|_{L^2(\Omega)}  + \delta^2F_M^2(t)). 
\]
Note that $\sigma^d(0)=0$,  integrating the resulting inequality in time, we have
\[
\begin{split}
\|\sigma^d (t)\|_{L^2(\Omega)} 
&\leq C_{22}\int_0^t (\|u_2^d(s)\|_{L^2(\Omega)}  + \delta^2F_M^2(s))ds.
\end{split}
\]
Together with \eqref{EstU^dGeq}, we have
\begin{equation}\label{EstSigma^dGeq}
\|\sigma^d (t)\|_{L^2(\Omega)}^2 \leq C_{23}\delta^3 \Big(\sum_{j=j_m}^N |\Csf_j| e^{\lambda_j t}+ \max(0,M-N) \Big(\max_{N+1\leq j\leq M}|\Csf_j|\Big) e^{\frac23 \nu_0\Lambda t}\Big)^3.
\end{equation}
The inequality \eqref{L2_NormU^d} follows from \eqref{EstU^dGeq} and \eqref{EstSigma^dGeq}. Proof of Proposition \ref{PropL2_NormU^d} is complete. 
\end{proof}

\subsection{Proof of Theorem \ref{ThmNonlinear}}
Note that 
\begin{equation}
\begin{split}
\|\vec u^M(t)\|_{L^2(\Omega)}^2 &=   \sum_{i=j_m}^M\Csf_i^2 e^{2\lambda_i t} \|\vec u_i\|_{L^2(\Omega)}^2 +  2 \sum_{j_m\leq i<j\leq M} \Csf_i \Csf_j e^{(\lambda_i+\lambda_j)t} \int_{\Omega} \vec u_i(\vec x) \cdot \vec u_j(\vec x) d\vec x.
\end{split}
\end{equation}
It can be seen that 
\[
\begin{split}
\|\vec u^M(t)\|_{L^2(\Omega)}^2 \geq  &\sum_{j=j_m}^M \Csf_j^2 e^{2\lambda_j t}\|\vec u_j\|_{L^2(\Omega)}^2 + 2\sum_{j_m+1\leq i<j\leq M} \Csf_i\Csf_j e^{(\lambda_i+\lambda_j)t}   \int_{\Omega} \vec u_i(\vec x) \cdot \vec u_j(\vec x) d\vec x\\
&\qquad-|\Csf_{j_m}|\|\vec u_{j_m}\|_{L^2(\Omega)} \Big(\sum_{j=j_m+1}^M|\Csf_j|  \|\vec u_j\|_{L^2(\Omega)}\Big) e^{(\lambda_{j_m}+\lambda_{j_m+1})t}.
\end{split}
\]
By Cauchy-Schwarz's inequality, we obtain 
\[
\begin{split}
2\sum_{j_m+1 \leq i<j\leq M} &\Csf_i\Csf_j e^{(\lambda_i+\lambda_j)t}   \int_{\Omega} \vec u_i(\vec x) \cdot \vec u_j(\vec x) d\vec x \\
&\geq- \sum_{j_m+1 \leq i<j\leq M} |\Csf_i||\Csf_j| e^{(\lambda_i+\lambda_j)t} \|\vec u_i\|_{L^2(\Omega)}\|\vec u_j\|_{L^2(\Omega)}\\
&\geq - e^{(\lambda_{j_m+1}+\lambda_{j_m+2})t} \Big(\sum_{j=j_m+1}^M |\Csf_j|\|\vec u_j\|_{L^2(\Omega)}\Big)^2.
\end{split}
\]
This yields
\[
\begin{split}
\|\vec u^M(t)\|_{L^2(\Omega)}^2 &\geq \sum_{j=j_m}^M \Csf_j^2 e^{2\lambda_j t}\|\vec u_j\|_{L^2(\Omega)}^2 -e^{(\lambda_{j_m+1}+\lambda_{j_m+2})t} \Big(\sum_{j=j_m+1}^M |\Csf_j| \|\vec u_j\|_{L^2(\Omega)}\Big)^2 \\
&\qquad - |\Csf_{j_m}| e^{(\lambda_{j_m}+\lambda_{j_m+1})t}\|\vec u_{j_m}\|_{L^2(\Omega)} \Big(\sum_{j=j_m+1}^M|\Csf_j|  \|\vec u_j\|_{L^2(\Omega)}.
\end{split}
\]
Due to the assumption \eqref{NormalizedCond_Navier_2}, we deduce that
\[
\begin{split}
\|\vec u^M(t)\|_{L^2(\Omega)}^2 &\geq \sum_{j=j_m}^M \Csf_j^2 e^{2\lambda_j t}\|\vec u_j\|_{L^2(\Omega)}^2 -\frac12 \Csf_{j_m}^2 e^{(\lambda_{j_m}+\lambda_{j_m+1})t} \|\vec u_{j_m}\|_{L^2(\Omega)}^2  \\
&\qquad - \frac14 \Csf_{j_m}^2 e^{(\lambda_{j_m+1}+\lambda_{j_m+2})t} \|\vec u_{j_m}\|_{L^2(\Omega)}^2.
\end{split}
\]
This yields 
\[\begin{split}
\|\vec u^M(t)\|_{L^2(\Omega)}^2 &\geq \Csf_{j_m}^2\Big(e^{2\lambda_{j_m} t}-\frac12 e^{(\lambda_{j_m}+\lambda_{j_m+1})t} -  \frac14 e^{(\lambda_{j_m+1}+\lambda_{j_m+2})t}\Big) \|\vec u_{j_m}\|_{L^2(\Omega)}^2  \\
&\qquad\qquad +\sum_{j=j_m+1}^M \Csf_j^2 e^{2\lambda_j t}\|\vec u_j\|_{L^2(\Omega)}^2.
\end{split}
\]
Notice that for all $t\geq 0$,
\[
e^{2\lambda_{j_m} t}-\frac12 e^{(\lambda_{j_m}+\lambda_{j_m+1})t} -  \frac14 e^{(\lambda_{j_m+1}+\lambda_{j_m+2})t} \geq \frac14 e^{2\lambda_{j_m} t}.
\]
Hence, we have 
\begin{equation}\label{L^2NormU_2^M}
\|\vec u^M(t)\|_{L^2(\Omega)} \geq C_{24} F_M(t),
\end{equation}
for all $t\leq \min(T^\delta, T^\star, T^{\star\star})$.

Let 
\[
\tilde \Csf(M) = \max_{N+1\leq j\leq M}\frac{|\Csf_j|}{|\Csf_{j_m}|}\geq 0.
\]
We recall the definition of $T^{\star}$ and $T^{\star\star}$ from \eqref{DefTstar} and the fact that $T^{\delta}$ satisfies uniquely $\delta F_M(T^\delta)=\epsilon_0$, 
provided that  $\epsilon_0$ is taken to be 
\begin{equation}\label{EqEpsilon}
\epsilon_0 < \min\Big( \frac{C_2\delta_0}{C_3},  \frac{C_2^2}{2C_4(1+ M \tilde\Csf(M))^3 }, \frac{C_{24}^2}{4C_4(1+ M\tilde \Csf(M))^3}  \Big).
\end{equation}
We then prove that 
\begin{equation}\label{T_deltaMin}
T^{\delta} \leq \min\{T^{\star}, T^{\star\star}\}.
\end{equation}
In fact, if $T^{\star} < T^{\delta}$, we have from \eqref{EnergyNormDelta} that
\[
\cE((\sigma^{\delta},\vec u^{\delta})(T^{\star})) \leq C_3 \delta F_M(T^{\star}) \leq C_3\delta F_M(T^{\delta}) = C_3\epsilon_0 <C_2\delta_0.
\]
And if $T^{\star\star}<T^{\delta}$, we have by \eqref{L2_NormU^d} and the definition of $T^{\delta}$ that
\begin{equation}\label{EstBoundSigmaDelta}
\begin{split}
&\|(\sigma^{\delta}, \vec u^{\delta})(T^{\delta})\|_{L^2(\Omega)}\\
 &\leq \delta \|(\sigma^M, \vec u^M)(T^{\delta})\|_{L^2(\Omega)} + \|(\sigma^d, \vec u^d)(T^{\delta})\|_{L^2(\Omega)}  \\
&\leq C_2 \delta F_M(T^{\delta}) + \sqrt{C_4}  \delta^{\frac32} \Big( \sum_{j=1}^N |\Csf_j| e^{\lambda_j T^{\delta}} + \max(0,M-N)\Big(\max_{N+1\leq j\leq M}|\Csf_j|\Big) e^{\frac23 \nu_0\Lambda T^{\delta}}\Big)^{\frac32}.
\end{split}
\end{equation}
Notice from \eqref{AssumeLambdaN} that for $N+1\leq j\leq M$,
\[
|\Csf_j| \delta e^{\frac23 \nu_0\Lambda T^{\delta}} < \frac{|\Csf_j|}{|\Csf_{j_m}|}( \delta |\Csf_{j_m}|  e^{\lambda_1 T^{\delta}})<  \frac{|\Csf_j|}{|\Csf_{j_m}|} \delta F_M(T^{\delta}) =  \frac{|\Csf_j|}{|\Csf_{j_m}|} \epsilon_0.
\] 
Then, it follows from \eqref{EstBoundSigmaDelta} that
\[
\begin{split}
\|(\sigma^{\delta}, \vec u^{\delta})(T^{\delta})\|_{L^2(\Omega)} &\leq C_2\delta F_M(T^{\delta}) + \sqrt{C_4}\delta^{\frac32}(1+ M\tilde \Csf(M))^{\frac32} F_M^{\frac32}(T^\delta) \\
&\leq C_2\epsilon_0+  \sqrt{C_4}(1+ M\tilde \Csf(M))^{\frac32}\epsilon_0^{\frac32}.
\end{split}
\]
Using \eqref{EqEpsilon} again, we deduce 
\[
\|(\sigma^{\delta}, \vec u^{\delta})(T^{\delta})\|_{L^2(\Omega)}< 2 C_2 \epsilon_0 = 2C_2 \delta F_M(T^{\delta}).
\]
which also contradicts  the definition of $T^{\star\star}$. 

Once we have \eqref{T_deltaMin}, we then get from \eqref{L2_NormU^d} and \eqref{L^2NormU_2^M} that 
\[
\begin{split}
\|\vec u^{\delta}(T^{\delta})\|_{L^2(\Omega)} &\geq \delta \|\vec u^M(T^{\delta})\|_{L^2(\Omega)} - \|\vec u^d(T^{\delta})\|_{L^2(\Omega)}\\
&\geq C_{24} \delta F_M(T^{\delta}) - \sqrt{C_4}  \delta^{\frac32} \sum_{j=1}^N |\Csf_j| e^{\lambda_j T^{\delta}} \\
&\qquad\quad-\sqrt{C_4}  \delta^{\frac32} \max(0,M-N)\Big(\max_{N+1\leq j\leq M}|\Csf_j|\Big) e^{\frac23 \nu_0\Lambda T^{\delta}}\Big)^{\frac32}.
\end{split}
\]
Therefore, 
\begin{equation}
\begin{split}
\|\vec u^{\delta}(T^{\delta})\|_{L^2(\Omega)} &\geq C_{24}\epsilon_0 - \sqrt{C_4}(1+ M\tilde \Csf(M))^{\frac32}\epsilon_0^{\frac32} \geq \frac{C_{24}\epsilon_0}2 >0.
\end{split}
\end{equation}
The inequality \eqref{BoundU_2Tdelta} is proven by taking $\delta_0$ satisfying \eqref{EqDelta_0},  $\epsilon_0$ satisfying \eqref{EqEpsilon} and $m_0=\frac12C_{24}$. This ends the proof of Theorem \ref{ThmNonlinear}.

\section*{Acknowledgments}

The author is deeply grateful to Prof. Jean-Marc Delort and Prof.  Olivier Lafitte  for their fruitful  discussions on this paper.  Thanks also go to Prof. Jeffrey Rauch for his advice on this study. The author wishes to thank Assoc. Prof. Quốc-Anh Ngô for his encouragement.  The author would also thank the hospitality of Centre International de Rencontre Math\'ematique and Universit\'e de Montr\'eal during the visit where  parts of this work were accomplished.  This work is supported by a grant from R\'egion \^Ile-de-France.

\appendix

\section{The precise value  of $\mu_c(k,\Xi)$}\label{FormulaMuK}
In this appendix, we prove Proposition \ref{PropMuC}(1). The equality \eqref{MuC_kMax} can be seen immediately from the definition of $\bB_{k,0,\mu}$. 

Note that the quotient 
\[
 \frac{\xi_-(\phi'(-1))^2+\xi_+(\phi'(1))^2}{\int_{-1}^1 ((\phi'')^2+2k^2(\phi')^2+k^4\phi^2) dx_2}
 \]
is bounded because of the embedding $H^2((-1,1))\hookrightarrow C^1((-1,1))$.
To prove \eqref{EqMuK}, let us consider the Lagrangian functional 
\begin{equation}
\cL_k(\phi, \beta) = \beta \Big(\int_{-1}^1((\phi'')^2+2k^2(\phi')^2+k^4\phi^2)dx_2 -1)- ( \xi_-(\phi'(-1))^2+\xi_+ (\phi'(1))^2)
\end{equation}
for any $\phi \in \tilde H^2((-1,1))$ and $\beta \neq 0$. Using Lagrange multiplier theorem, the extremal points of the quotient 
\[
 \frac{\xi_-(\phi'(-1))^2+\xi_+(\phi'(1))^2}{\int_{-1}^1 ((\phi'')^2+2k^2(\phi')^2+k^4\phi^2) dx_2}
\]
are necessarily  the stationary points of $(\phi_k,\beta_k)$ of $\cL_k$, which satisfy  
\begin{equation}\label{Psi_kConstraintAt1}
\int_{-1}^1((\phi_k'')^2+2k^2(\phi_k')^2+k^4\phi_k^2)dx_2 =1, 
\end{equation}
and 
\begin{equation}\label{EqConstraintPsi_k}
\beta_k \int_{-1}^1 (\phi_k'' \omega''+2k^2 \phi_k'\omega'+ k^4\phi_k\omega) dx_2 =\xi_- \phi_k'(-1) \omega'(-1) + \xi_+ \phi_k'(1) \omega'(1)
\end{equation}
for all $\omega \in \tilde H^2((-1,1))$. We obtain from \eqref{EqConstraintPsi_k} after  taking integration by parts that 
\[
\phi_k^{(4)}-2k^2 \phi_k''+k^4\phi_k =0,
\]
and
\[
\begin{cases}
\beta_k(\phi_k'''(1)+2k^2 \phi_k'(1))\omega(1) = 0, \\
(\beta_k \phi_k''(1) -\xi_+\phi_k'(1))\omega'(1)=0,\\
\beta_k( \phi_k'''(-1)+2k^2 \phi_k'(-1))\omega(-1) = 0, \\
(\beta_k\phi_k''(-1) + \xi_- \phi_k'(-1))\omega'(-1)=0,
\end{cases}
\]
for all $\omega \in \tilde H^2((-1,1))$. This yields
\begin{equation}\label{BoundaryPhi_k}
\begin{cases}
\beta_k \phi_k''(1) -\xi_+\phi_k'(1)=0,\\
\beta_k \phi_k''(-1) +\xi_-\phi_k'(-1)=0.
\end{cases}
\end{equation}

Hence, $\phi_k$ is of the form 
\[
\phi_k(x_2) =(Ax_2+B) \sinh(kx_2) +(Cx_2+D)\cosh(kx_2),
\]
with $A,B,C,D$ are four constants such that $A^2+B^2+C^2+D^2>0$.  Since $\phi_k\in \tilde H^2((-1,1))$, we get 
\[
\begin{cases}
(A+B)\sinh k + (C+D)\cosh k=0,\\
(-A+B)\sinh(-k)+  (-C+D) \cosh(-k)=0.
\end{cases}
\]
It yields
\begin{equation}\label{GammaEqualDC-AB}
C= -B \tanh(k) \quad\text{and} \quad D=  -A\tanh(k).
\end{equation}
We then compute 
\[
\phi_k'(x_2)= (A+kD+kCx_2) \sinh(kx_2) + (C+kB+kAx_2) \cosh(kx_2)
\]
and 
\[
\phi_k''(x_2) = (2kC+k^2B+ k^2Ax_2) \sinh(kx_2) + (2kA+ k^2D +k^2 Cx_2) \cosh(kx_2).
\]
Substituting these formulas into \eqref{BoundaryPhi_k}, we have 
\[
\begin{cases}
\beta_k \Big( (2kC+k^2(B+A)) \sinh k + (2kA +k^2(D+C)) \cosh k \Big)\\
\qquad\quad= \xi_+ \Big( (A+k(D+C)) \sinh k + (C+k(B+A)) \cosh k \Big),\\
\beta_k \Big( (2kC+ k^2(B-A)) \sinh(- k) + (2kA+ k^2(D-C))\cosh(-k) \Big)\\
\qquad\quad= - \xi_- \Big( (A+k(D-C)) \sinh(-k) + (C+k(B-A))\cosh(-k)).
\end{cases}
\]
Thanks to  \eqref{GammaEqualDC-AB}, that reduces to 
\[
\begin{cases}
2k \beta_k (C\sinh k+A\cosh k)  \\
\qquad =  \xi_+ \Big( A+k(D+C) \sinh k + (C+k(B+A)) \cosh k \Big),\\
2k \beta_k (-C\sinh k+ A\cosh k) \\
\qquad=  \xi_-\Big((A+k(D-C)) \sinh(k) - (C+k(B-A))\cosh(k) \Big).
\end{cases}
\]
Equivalently, 
\begin{equation}\label{SystA-B_hyper}
\begin{cases}
2k \beta_k \Big( A+B + (A-B)\cosh(2 k)\Big) =  \xi_+ \Big( (A-B)\sinh(2k) + 2k (A+B) \Big),\\
2k \beta_k \Big( A-B + (A+B)\cosh(2 k)\Big) =\xi_-  \Big( (A+B) \sinh(2k) +2 k(A-B) \Big).
\end{cases}
\end{equation}
Then, $(A,B)$ is a solution of the following system
 \begin{equation}\label{SystA-Btransform}
\begin{cases}
A \Big(2k(1+\cosh(2k))\beta_k - \xi_+ (2k+\sinh(2k)) \Big) \\
\qquad=B \Big( 2k(\cosh(2k)-1) \beta_k- \xi_+ (\sinh(2k)-2k) \Big), \\
A \Big(2k(1+\cosh(2k))\beta_k- \xi_- (2k+\sinh(2k)) \Big) \\
\qquad=-B \Big( 2k(\cosh(2k)-1)  \beta_k- \xi_- (\sinh(2k)-2k) \Big). \\
\end{cases}
\end{equation}
System \eqref{SystA-Btransform} admits a nontrivial solution if and only 
\begin{equation}\label{EqBeta_k}
\begin{split}
&\Big(2k(1+\cosh(2k))\beta_k - \xi_+ (2k+\sinh(2k)) \Big) \\
&\qquad\quad \times \Big( 2k(\cosh(2k)-1)  \beta_k- \xi_- (\sinh(2k)-2k) \Big)\\
&=-\Big(2k(1+\cosh(2k))\beta_k- \xi_- (2k+\sinh(2k)) \Big) \\
&\qquad\quad \times \Big( 2k(\cosh(2k)-1) \beta_k- \xi_+ (\sinh(2k)-2k) \Big).
\end{split}
\end{equation}
We rewrite Eq. \eqref{EqBeta_k} as a quadratic equation of $\beta_k$, that is 
\begin{equation}\label{QuadraticBeta_k}
\begin{split}
4k^2 (\cosh^2(2k) -1) \beta_k^2&- 2k( \sinh(2k)\cosh(2k) -2k)(\xi_+ +\xi_-) \beta_k\\
&\qquad+(\sinh^2(2k)-4k^2) \xi_+ \xi_-   = 0. 
\end{split}
\end{equation}
The discriminant is
\[
\begin{split}
\Delta_{k,\Xi} &= k^2( \sinh(2k)\cosh(2k)-2k)^2 (\xi_+ +\xi_-)^2 \\
&\qquad- 4k^2 (\cosh^2(2k)-1)(\sinh^2(2k)-4k^2) \xi_+ \xi_- \\
&=  k^2 ( \sinh(2k) -2k\cosh(2k))^2 (\xi_++ \xi_-)^2 \\
&\qquad+  k^2 \sinh^2(2k)(\sinh^2(2k)-4k^2) (\xi_+ -\xi_-)^2.
\end{split}
\]
 Because $\tanh(2k)<2k$ for all $k>0$ and  $\xi_+^2+\xi_-^2> 0$, we have $\Delta_{k,\xi}$ is always positive. Hence, we have that \eqref{QuadraticBeta_k} has two roots
\[
\beta_{k,\pm}=\frac{k( \sinh(2k)\cosh(2k) -2k)(\xi_+ +\xi_-)  \pm \sqrt{\Delta_{k,\Xi}}}{4k^2 \sinh^2(2k) }.
\]
We take the higher value $\beta_{k,+}>0$ and then solve the system \eqref{SystA-Btransform} as $\beta_k=\beta_{k,+}$. 

If $\xi_-\geq \xi_+$, we have
\[
\begin{split}
&4\sinh^2k \Big( 2k(\cosh(2k)-1) \beta_{k,+}- \xi_+ (\sinh(2k)-2k) \Big)\\
&= (\sinh(2k)\cosh(2k)-2k) (\xi_- -\xi_+) - 2(\sinh(2k)-2k\cosh(2k))\xi_+ + \frac1k \sqrt{\Delta_{k,\Xi}} \\
&>0.
\end{split}
\]
Then, we obtain from $\eqref{SystA-Btransform}_1$ that 
\begin{equation}\label{EqA_gamma_k}
B= A \frac{2k(1+\cosh(2k))\beta_{k,+} - \xi_+ (2k+\sinh(2k)) }{ 2k(\cosh(2k)-1)  \beta_{k,+}- \xi_+ (\sinh(2k)-2k)} =: Aa_{k,\Xi}.
\end{equation}
So that 
\[
(A,B,C,D) = A(1,a_{k,\Xi}, -a_{k,\Xi}\tanh k, -\tanh k)
\]
with $A\neq 0$ and $\phi_k(x_2) = A z_k(x_2)$, with 
\[
 z_k(x_2) = (x_2+a_{k,\Xi})\sinh(kx_2) -\tanh k (a_{k,\Xi} x_2+1) \cosh(kx_2).
\]
We find $A$ from \eqref{Psi_kConstraintAt1}, such that
\begin{equation}\label{EqAz_k}
A^2 \int_{-1}^1 ((z_k'')^2+2k^2 (z_k')^2+k^4z_k^2) dx_2=1.
\end{equation}

If $0<\xi_-<\xi_+$, that will imply 
\[
2k(\cosh(2k)-1)\beta_{k,+} -\xi_- (\sinh(2k)-2k) >0.
\]
We further get from $\eqref{SystA-Btransform}_2$ that 
\[
B= -A \frac{2k(1+\cosh(2k))\beta_{k,+} -\xi_- (2k+\sinh(2k)) }{ 2k(\cosh(2k)-1)  \beta_{k,+} -\xi_- (\sinh(2k)-2k)}=: -A b_{k,\Xi}.
\]
So that, we have 
\[
(A,B,C,D) = A(1,-b_{k,\Xi}, b_{k,\Xi}\tanh k, -\tanh k)
\]
with $A\neq 0$ and $\phi_k(x_2) = A w_k(x_2)$, with 
\[
w_k(x_2) =  (x_2-b_{k,\Xi})\sinh(kx_2) +  (b_{k,\Xi}x_2-1)\tanh k\cosh(kx_2).
\]
We still find $A$ from \eqref{Psi_kConstraintAt1}, 
\begin{equation}\label{EqAw_k}
A^2 \int_{-1}^1 ((w_k'')^2+2k^2 (w_k')^2+k^4w_k^2) dx_2=1.
\end{equation}

We have just shown that  
\[
\begin{split}
\mu_c(k,\Xi) &= \max_{\phi \in \tilde H^2((-1,1))} \frac{ \xi_-(\phi'(-1))^2+\xi_+ (\phi'(1))^2}{\int_{-1}^1 ((\phi'')^2+2k^2(\phi')^2+k^4\phi^2)dx_2} \\
&= \frac1{4k \sinh^2(2k) } 
\left( \begin{split} 
&( \sinh(2k)\cosh(2k) -2k)(\xi_+ +\xi_-) \\
&+ \left(\begin{split} 
&( \sinh(2k) -2k\cosh(2k))^2 (\xi_++ \xi_-)^2  \\
& + \sinh^2(2k)(\sinh^2(2k)-4k^2) (\xi_+ -\xi_-)^2
\end{split}\right)^{\frac12}
\end{split}\right).
\end{split}
\]
That variational problem is attained by functions $\phi_k(x_2) = A z_k(x_2)$, where $A$ satisfies \eqref{EqAz_k} or $\phi_k(x_2)=Aw_k(x_2)$, where $A$ satisfies \eqref{EqAw_k}.
The equality \eqref{EqMuK} is shown and  the proof of the first part of Proposition \ref{PropMuC} then follows.

\section{Asymptotic behavior of $\mu_c(k,\Xi)$ in low/high regime of wave number}\label{AppLimitMuK}
Let us prove Proposition \ref{PropMuC}(2). Clearly, we have that $\mu_c(k,\Xi)$ is a decreasing function in $k>0$. It yields \eqref{Mu_Climit}.

We first consider $k\to 0$. Let us recall the Taylor's expansion of $\sinh(2k)$ and $\cosh(2k)$. We have 
\[
\sinh(2k) = 2k + \frac43 k^3 + \frac4{15}k^5 +O(k^6), \quad\text{and}\quad \cosh(2k) =1+ 2k^2 + \frac23 k^4+ O(k^5).
\]
We deduce that 
\[
\frac{\sinh(2k)\cosh(2k)-2k}{4k\sinh^2(2k)} = \frac{ \frac16 + \frac2{15} k^2 +O(k^3)}{\frac12 +\frac23 k^2+O(k^3)} = \frac13 -\frac8{15}k^2+ O(k^3),
\]
that 
\[
\frac{\sinh(2k)-2k\cosh(2k)}{4k\sinh^2(2k)} = \frac{ -\frac83 -\frac{16}{15} k^2+ O(k^3)}{16+ \frac{64}3k^2+ O(k^3)}= -\frac16+\frac7{45} k^2 + O(k^3)
\]
and that 
\[
\frac{\sinh^2(2k)(\sinh^2(2k)-4k^2)}{16k^2 \sinh^4(2k)} = \frac{\frac{16}3 + \frac{128}{45}k^2+O(k^2)}{64+ \frac{256}3 k^2+O(k^2)} = \frac1{12} -\frac1{15} k^2+O(k^3).
\]
We deduce that
\begin{equation}\label{1stOrderLimitMuK}
\lim_{k\to 0} \mu_c(k,\Xi) = \frac13(\xi_+ +\xi_-) +  \sqrt{ \frac1{36}(\xi_++\xi_-)^2+ \frac1{12}(\xi_+-\xi_-)^2}.
\end{equation}
That will imply \eqref{EqMuC-MuC_k}, i.e. 
\[
\mu_c^s(\Xi) = \frac13 \Big(\xi_+ +\xi_- +\sqrt{\xi_+^2 -\xi_+\xi_- +\xi_-^2}\Big).
\]
 Furthermore, we have that 
\begin{equation}\label{2ndOrderLimitMuK}
\lim_{k\to 0} \frac{\mu_c(k,\Xi)-\mu_c^s(\Xi)}{k^2}= -\frac2{15} \Big( 4(\xi_++\xi_-) + \frac{4\xi_+^2-\xi_+\xi_- +4\xi_-^2}{\sqrt{\xi_+^2-\xi_+\xi_-+\xi_-}} \Big).
\end{equation}
Two limits \eqref{1stOrderLimitMuK} and \eqref{2ndOrderLimitMuK} help us to get \eqref{LimitMuK}.

For high wave number, i.e. $k\to +\infty$, we can see that
\[
\begin{split}
\frac{\sinh(2k)\cosh(2k)-2k}{\sinh^2(2k)} 
&= \frac{1-e^{-8k}-8k e^{-4k}}{1+e^{-8k}-2e^{-4k}} \leq 2,
\end{split}
\]
that 
\[
\begin{split}
\frac{\sinh(2k)-2k\cosh(2k)}{\sinh^2(2k)}  = \frac12 \frac{1-2k-(1+2k)e^{-4k}}{e^{2k}+e^{-6k}-2e^{-2k}} \leq 1.
\end{split}
\]
Hence, 
\[
\begin{split}
\mu_c(k,\Xi) &= \frac1{4k} 
\left( \begin{split} 
& \frac{\sinh(2k)\cosh(2k) -2k}{\sinh^2(2k)}(\xi_+ +\xi_-) \\
&+ \left(\begin{split} 
&\Big( \frac{\sinh(2k) -2k\cosh(2k)}{\sinh^2(2k)} \Big)^2 (\xi_++ \xi_-)^2  \\
& + \Big(1-\frac{4k^2}{\sinh^2(2k)}\Big) (\xi_+ -\xi_-)^2
\end{split}\right)^{\frac12}
\end{split}\right) \\
&\leq \frac1{4k}\Big( 2(\xi_++\xi_-) + \sqrt{2(\xi_+^2+\xi_-^2)} \Big).
\end{split}
\]
That implies \eqref{UpperBoundMuK}. The proof of the second assertion of Proposition \ref{PropMuC} is complete.

\section{Proof of Proposition \ref{PropMuC}(3)}\label{AppMuC}

In this appendix, we prove Proposition \ref{PropMuC}(3). We first show that 
\begin{equation}\label{EqSupMuC}
\mu_c^s(\Xi) =\sup\limits_{\phi\in \tilde H^2((-1,1))} \frac{\xi_-(\phi'(-1))^2+\xi_+(\phi'(1))^2}{\int_{-1}^1 (\phi'')^2 dx_2}
\end{equation}
Indeed, let 
\[
\tilde \mu_c(\Xi) :=\sup\limits_{\phi\in \tilde H^2((-1,1))} \frac{\xi_-(\phi'(-1))^2+\xi_+(\phi'(1))^2}{\int_{-1}^1 (\phi'')^2 dx_2}
\]
and then prove that $\mu_c^s(\Xi)=\tilde \mu_c(\Xi)$. Clearly, we have $\mu_c(k,\Xi) \leq \tilde \mu_c(\Xi)$ for all $k\in \R\setminus \{0\}$. It yields $\mu_c^s(\Xi) \leq \tilde \mu_c(\Xi)$.  It suffices to show that $\tilde \mu_c(\Xi) \geq \mu_c^s(\Xi)$.
For any $\varepsilon>0$, from the definition of $\mu_c(\Xi)$,  we fix a function $\phi_\varepsilon \in \tilde H^2((-1,1))$ such that 
\[
 \frac{\xi_-(\phi_\varepsilon'(-1))^2+\xi_+(\phi_\varepsilon'(1))^2}{\int_{-1}^1 (\phi_\varepsilon'')^2 dx_2}\geq \tilde \mu_c(\Xi)  -\varepsilon. 
\]
Let $k\neq 0$ be small enough, we then obtain 
\[
 \frac{\xi_-(\phi_\varepsilon'(-1))^2+\xi_+(\phi_\varepsilon'(1))^2}{\int_{-1}^1 ((\phi_\varepsilon'')^2+2k^2(\phi_\varepsilon')^2+k^4\phi_\varepsilon^2) dx_2} > \frac{\xi_-(\phi_\varepsilon'(-1))^2+\xi_+(\phi_\varepsilon'(1))^2}{\int_{-1}^1 (\phi_\varepsilon'')^2 dx_2} - \varepsilon.  
\]
That implies
\[
\mu_c(k,\Xi) > \tilde\mu_c(\Xi) - 2\varepsilon. 
\]
We deduce that $\tilde \mu_c(\Xi) =\sup_{k\in \R\setminus\{0\}}\mu_c(k,\Xi)$, i.e. \eqref{EqSupMuC}.

Then, we show that 
\begin{equation}\label{FormulaMuC}
\max\limits_{\phi\in \tilde H^2((-1,1))} \frac{\xi_-(\phi'(-1))^2+\xi_+(\phi'(1))^2}{\int_{-1}^1 (\phi'')^2 dx_2}
 = \frac13\Big( \xi_++\xi_- +\sqrt{\xi_+^2 -\xi_+\xi_- +\xi_-^2}\Big).
\end{equation}
Let us consider the Lagrangian functional 
\begin{equation}
\cL_0(\phi, \beta) = \beta\Big( \int_{-1}^1(\phi'')^2dx_2-1\Big) - \xi_-(\phi'(-1))^2-\xi_+ (\phi'(1))^2.
\end{equation}
for any $\phi \in \tilde H^2((-1,1))$ and $\beta \neq 0$. Owing to Lagrange multiplier theorem, the extrema points of the quotient 
\[
 \frac{\xi_-(\phi'(-1))^2+\xi_+(\phi'(1))^2}{\int_{-1}^1 (\phi'')^2 dx_2}
\]
are necessarily the stationary points  $(\phi_0,\beta_0)$ of $\cL_0$, which satisfy 
\begin{equation}\label{Psi_0ConstraintAt1}
\int_{-1}^1(\phi_0'')^2dx_2 =1, 
\end{equation}
and 
\begin{equation}\label{EqConstraintPsi_0}
\beta_0 \int_{-1}^1 \phi_0'' \omega''dx_2 -  (\xi_- \phi_0'(-1) \omega'(-1) + \xi_+ \phi_0'(1) \omega'(1))=0
\end{equation}
for all $\omega \in \tilde H^2((-1,1))$. We obtain from \eqref{EqConstraintPsi_0} after taking integration by parts that 
\[
\phi_0^{(4)} =0 \quad\text{on }(-1,1). 
\]
and 
\begin{equation}\label{BoundaryPhi_0}
\begin{cases}
\phi_0''(1) =\xi_+\phi_0'(1),\\
\phi_0''(-1)=-\xi_-\phi_0'(-1).
\end{cases}
\end{equation}


Hence, $\phi_0$ is of the form 
\[
\phi_0(x_2) =(x_2^2-1)(Ax_2+B).
\]
Substituting this form of $\phi_0$ into \eqref{BoundaryPhi_0}, we have that 
\[
\begin{cases}
\beta_0(3A+B) =  \xi_+(A+B), \\
\beta_0(3A-B) =  \xi_-(A-B).
\end{cases}
\]
Hence, 
\begin{equation}\label{SystAB_0}
\begin{cases}
A(3\beta_0-\xi_+) + B(\beta_0 -\xi_+) =0,\\
A(3\beta_0 -\xi_-) - B(\beta_0-\xi_-) =0.
\end{cases}
\end{equation}
System \eqref{SystAB_0} admits a nontrivial solution $(A,B)$ if and only if 
\[
(3\beta_0-\xi_+)(\beta_0-\xi_-) + (3\beta_0 -\xi_-)(\beta_0 -\xi_+)=0.
\]
It yields
\begin{equation}\label{EqBeta_0}
3 \beta_0^2 - 2(\xi_++\xi_-) \beta_0+\xi_-\xi_+=0.
\end{equation}
The discriminant of \eqref{EqBeta_0} is
\[
\Delta_{0,\xi}= (\xi_+ +\xi_-)^2-3\xi_-\xi_+ = \xi_+^2-\xi_+\xi_- +\xi_-^2 >0.
\]
Then, Eq. \eqref{EqBeta_0}  has two  roots
\[
\beta_{0,\pm}= \frac13 \Big(\xi_++ \xi_-\pm  \sqrt{\xi_+^2-\xi_+\xi_- +\xi_-^2}\Big).
\]
We take the higher value $\beta_{0,+}$.  As  $\beta_0=\beta_{0,+}$, we have  from $\eqref{SystAB_0}_2$ that 
\[
A(3\beta_{0,+}-\xi_-) = B(\beta_{0,+}-\xi_-).
\]
It is obvious that 
\[
3\beta_{0,+}-\xi_- = \xi_+ + \sqrt{\xi_+^2-\xi_+\xi_-+\xi_-^2} >0.
\]
Then we have
\[
A= B \frac{\beta_{0,+}-\xi_-}{3\beta_{0,+}-\xi_-} 
\]
and 
\[
 \phi_0(x_2)= B z_0(x_2), \quad\text{with } z_0(x_2)=\Big(  \frac{\beta_{0,+}-\xi_-}{3\beta_{0,+}-\xi_-} x_2+1\Big) (x_2^2-1).
\]
We continue using \eqref{Psi_0ConstraintAt1} to find a non-zero $B$. This yields
\[
B^2 \int_{-1}^1 (z_2''(x_2))^2 dx_2=1.
\]
That is equivalent to 
\[
8B^2 \Big( 3  \frac{(\beta_{0,+}-\xi_-)^2}{(3\beta_{0,+}-\xi_-)^2}+4\Big) = 1,
\]
this yields 
\[
B = \pm \frac1{2\sqrt 2} \frac{3\beta_{0,+}-\xi_-}{ \sqrt{39\beta_{0,+}^2 -30\beta_{0,+}\xi_-+7\xi_-^2}}.
\]

That means, we observe
\[
\max_{\phi\in \tilde H^2((-1,1))} \frac{\xi_-(\phi'(-1))^2+\xi_+(\phi'(1))^2}{\int_{-1}^1(\phi'')^2 dx_2} =  \frac13 \Big(\xi_++ \xi_-+ \sqrt{\xi_+^2-\xi_+\xi_- +\xi_-^2}\Big).
\]
That variational problem is attained by functions 
\[
\phi_0(x_2) = \pm\frac1{2\sqrt 2}\frac{3\beta_{0,+}-\xi_-}{ \sqrt{39\beta_{0,+}^2 -30\beta_{0,+}\xi_-+7\xi_-^2}} \Big( \frac{\beta_{0,+}-\xi_-}{3\beta_{0,+}-\xi_-}x_2+1\Big)(x_2^2-1).
\]
Combining \eqref{EqSupMuC} and \eqref{FormulaMuC}, we obtain Proposition \ref{PropMuC}(3).

\section{Comments on paper of Ding, Zi and Li}

In \cite{DJL20}, the authors Ding, Zi and Li construct an approximate solution generated by the maximal normal mode,  $(\sigma^a, \vec u^a, q^a)(t,x) = \delta e^{\lambda_1(k) t}\vec U_1(\vec x)$ with $k$ being fixed such that $\frac{2\Lambda}3 <\lambda_1(k)<\Lambda$. Applying Proposition \ref{PropLocalSolution},  the  nonlinear equations \eqref{EqPertur}-\eqref{BoundLinearized} with the initial data 
\[
(\sigma^{\delta}, \vec u^{\delta}, q^{\delta})(0)= (\sigma^a,\vec u^a,q^a)(0).
\] 
 admits a strong solution $(\sigma^{\delta},\vec u^{\delta}) \in C^0([0,T^{\max}),H^1 \times H^2)$ with an associated pressure $q^{\delta} \in C^0([0,T^{\max}), L^2)$. Let  $T^{\delta}$ such that $\delta e^{\lambda_1 T^{\delta}}=\epsilon_0 \ll 1$.  We define
\[
\begin{split}
T^{\star} &:= \sup\Big\{t\in (0,T^{\max})| \cE(\sigma^{\delta}(t), \vec u^{\delta}(t)) \leq C\delta_0\}>0,\\
T^{\star\star} &:= \sup \{t\in (0,T^{\max})| \|(\sigma^{\delta},\vec u^{\delta})(t)\|_{L^2(\Omega)} \leq C \delta e^{\lambda_1t}\Big\}>0.
\end{split}
\]
Then for all $t\leq \min\{T^{\delta}, T^{\star}, T^{\star\star}\}$, we have 
\[
\cE^2(\sigma^{\delta}(t), \vec u^{\delta}(t)) +\|\partial_t \vec u^{\delta}(t)\|_{L^2(\Omega)}^2+\int_0^t \|\nabla \partial_t \vec u^{\delta}(\tau)\|_{L^2(\Omega)}^2 d\tau \leq C \delta^2e^{2\lambda_1t}.
\]

In \cite[Proposition 5.2]{DJL20}, they claim that the difference functions 
\[
(\sigma^d, \vec u^d, q^d) = (\sigma^{\delta}, \vec u^{\delta}, q^{\delta})-(\sigma^a, \vec u^a, q^a)
\] enjoy 
\begin{equation}\label{IneProp52}
\|(\sigma^d, \vec u^d)\|_{L^2(\Omega)}^2 \leq C \delta^3 e^{3\lambda_1 t}
\end{equation}
for all $\mu >0$. We believe that \eqref{IneProp52} needs to be corrected, not for all $\mu>0$. Precisely, we are in doubt about inequality \textbf{(137)} in that paper, that is for all $t\leq \min\{T^{\delta}, T^{\star}, T^{\star\star}\}$, 
\begin{equation}\label{Ine137}
\begin{split}
&\|\sqrt{\rho_0+\sigma^{\delta}(t)} \partial_t \vec u^d(t)\|_{L^2(\Omega)}^2 +\Lambda\mu \|\nabla\vec u^d(t)\|_{L^2(\Omega)}^2+ \mu\int_0^t \|\nabla \partial_t \vec u^d(s)\|_{L^2(\Omega)}^2  ds \\
&\leq \Lambda \Big( \int_0^t\|\sqrt{\rho_0 +\sigma^{\delta}(s)}\vec u^d(s)\|_{L^2(\Omega)}^2+\Lambda\mu\|\nabla \vec u^d(s)\|_{L^2(\Omega)}^2\Big) \\
&\qquad\qquad +\Lambda\|\sqrt{\rho_0+\sigma^{\delta}(t)} \vec u^d(t)\|_{L^2(\Omega)}^2+ C\delta^3 e^{3\lambda_1 t}. 
\end{split}
\end{equation}
Due to \eqref{Ine137} and the following inequality 
\begin{equation}\label{Ine135}
\begin{split}
&\frac{d}{dt}\|\sqrt{\rho_0 +\sigma^{\delta}(t)}\vec u^d(t)\|_{L^2(\Omega)}^2\\ &=2\int_{\Omega} (\rho_0+\sigma^{\delta}(t)) \vec u^d(t) \cdot \partial_t \vec u^d(t) d\vec x + \int_{\Omega} \partial_t \sigma^{\delta}(t)|\vec u^d(t)|^2 d\vec x \\
 &\leq  \frac1{\Lambda} \|\sqrt{\rho_0+\sigma^{\delta}(t)} \partial_t \vec u^d(t)\|_{L^2(\Omega)}^2 +\Lambda \|\sqrt{\rho_0+\sigma^{\delta}(t)} \vec u^d(t)\|_{L^2(\Omega)}^2 + C\delta^3  e^{3\lambda_1 t},
\end{split}
\end{equation}
it is claimed in \cite[\textbf{(138)}]{DJL20} that 
\begin{equation}\label{Ine138}
\begin{split}
&\frac{d}{dt} \|\sqrt{\rho_0 +\sigma^{\delta}(t)}\vec u^d(t)\|_{L^2(\Omega)}^2 + \Big(\|\sqrt{\rho_0+\sigma^{\delta}(t)} \partial_t\vec u^d(t)\|_{L^2(\Omega)}^2 +\Lambda\mu \|\nabla\vec u^d(t)\|_{L^2(\Omega)}^2 \Big)\\
&\leq  \Lambda \int_0^t \Big( \|\sqrt{\rho_0+\sigma^{\delta}(s)} \partial_t\vec u^d(s)\|_{L^2(\Omega)}^2 +\Lambda\mu \|\nabla\vec u^d(s)\|_{L^2(\Omega)}^2   \Big) ds  \\
&\qquad\qquad +\Lambda\|\sqrt{\rho_0+\sigma^{\delta}(t)} \vec u^d(t)\|_{L^2(\Omega)}^2+ C\delta^3 e^{3\lambda_1 t}. 
\end{split}
\end{equation}
The inequality \eqref{IneProp52} is followed by applying Gronwall's inequality to \eqref{Ine138}. 

We shall explain the arguments  of \eqref{Ine137} in \cite{DJL20}. First, we still have
 \begin{equation}
\begin{split}
&\|\sqrt{\rho_0+\sigma^{\delta}(t)} \partial_t u^d(t)\|_{L^2(\Omega)}^2 + 2 \mu\int_0^t \|\nabla \partial_t \vec u^d(s)\|_{L^2(\Omega)}^2  ds \\
&\qquad-2\int_0^t \int_{2\pi L\bT} ( \xi_+ |u_1^d(s,x_1,1)|^2 +\xi_- |u_1^d(s, x_1,-1)|^2)  dx_1ds\\
&= \int_{\Omega} g\rho_0' |u_2^d(t)|^2 d\vec x +\Big( \int_{\Omega}(\rho_0+\sigma^{\delta}(t))|\partial_t\vec  u^d(t)|^2 d\vec x\Big)\Big|_{t=0}\\
&\qquad+ \int_0^t \int_{\Omega}(2\partial_t f^{\delta}(s)+2g \vec u^{\delta}(s)\cdot \nabla \sigma^{\delta}(s) e_2-\partial_t\sigma^{\delta}(s) \partial_t \vec u^d(s))\cdot \partial_t \vec u^d(s) ds.
\end{split}
\end{equation}
We  estimate
 \begin{equation}
\begin{split}
&\|\sqrt{\rho_0+\sigma^{\delta}(t)} \partial_t \vec u^d(t)\|_{L^2(\Omega)}^2 + 2 \mu\int_0^t \|\nabla \partial_t \vec u^d(s)\|_{L^2(\Omega)}^2  ds \\
&\qquad-2\int_0^t \int_{2\pi L\bT} ( \xi_+ |u_1^d(s,x_1,1)|^2 +\xi_- |u_1^d(s, x_1,-1)|^2)  dx_1ds\\
&\leq  \int_{\Omega} g\rho_0' |u_2^d(t)|^2 d\vec x + C \delta^3 e^{3\lambda_1t}.
\end{split}
\end{equation}
That implies 
\begin{equation}\label{IneU_tD_Sharp}
\begin{split}
&\|\sqrt{\rho_0+\sigma^{\delta}(t)} \partial_t \vec u^d(t)\|_{L^2(\Omega)}^2 + 2 \mu\int_0^t \|\nabla \partial_t \vec u^d(s)\|_{L^2(\Omega)}^2  ds \\
&\qquad\qquad -2 \int_0^t \int_{2\pi L\bT} ( \xi_+ |\partial_t u_1^d(s,x_1,1)|^2 +\xi_- |\partial_t u_1^d(s, x_1,-1)|^2)  dx_1ds\\
&\leq  \Lambda^2 \int_{\Omega}(\rho_0+\sigma^{\delta}(t))|\vec u^d(t)|^2 d\vec x+ \Lambda\mu \int_{\Omega}|\nabla \vec u^d(t)|^2 d\vec x \\
&\qquad\qquad -\Lambda \int_{2\pi L\bT} ( \xi_+ |u_1^d(t,x_1,1)|^2 +\xi_- |u_1^d(t, x_1,-1)|^2)  dx_1+ C\delta^3e^{3\lambda_1 t}.
\end{split}
\end{equation}
By using the inequality
\[
\Lambda\mu\|\nabla\vec u^d\|_{L^2(\Omega)}^2 \leq \Lambda^2\mu \int_0^t\|\nabla \vec u^d(s)\|_{L^2(\Omega)}^2 ds + \mu \int_0^t\|\nabla \partial_t\vec u^d(s)\|_{L^2(\Omega)}^2ds
\]
and the identity
\[
\begin{split}
&\Lambda\int_{2\pi L\bT} ( \xi_+ |u_{1}^d(t,x_1,1)|^2 +\xi_- |u_{1}^d(t, x_1,-1)|^2)  dx_1 \\
&= \Lambda^2 \int_0^t\int_{2\pi L\bT} ( \xi_+ |u_{1}^d(s,x_1,1)|^2 +\xi_- |u_{1}^d(s, x_1,-1)|^2)  dx_1ds \\
&\qquad + \int_0^t \int_{2\pi L\bT} ( \xi_+ |\partial_t u_1^d(s,x_1,1)|^2 +\xi_- |\partial_t u_1^d(s, x_1,-1)|^2)  dx_1ds\\
&\qquad- \int_0^t\int_{2\pi L\bT} ( \xi_+ |\Lambda u_{1}^d- \partial_t u_1^d|^2(s,x_1,1) +\xi_- |\Lambda u_{1}^d- \partial_t u_1^d|^2(s, x_1,-1))  dx_1ds,  
\end{split}
\]
it is obtained from \eqref{IneU_tD_Sharp} that (see \textbf{(134)} in \cite{DJL20})
\begin{equation}\label{Ine134}
\begin{split}
&\|\sqrt{\rho_0+\sigma^{\delta}(t)} \partial_t \vec u^d(t)\|_{L^2(\Omega)}^2 \\
&\quad + \frac12 \Lambda\Big( \mu \|\nabla \vec u^d(t)\|_{L^2(\Omega)}^2  - \int_{2\pi L\bT} ( \xi_+ |u_1^d(t,x_1,1)|^2 +\xi_- |u_1^d(t, x_1,-1)|^2)  dx_1\Big) \\
&\quad+ \frac12 \int_0^t \Big( \mu\|\nabla \partial_t \vec u^d(s)\|_{L^2(\Omega)}^2 - \int_{2\pi L\bT} ( \xi_+ | \partial_t u_1^d(s,x_1,1)|^2 +\xi_- | \partial_t u_1^d(s, x_1,-1)|^2)  dx_1\Big)ds \\
&\leq \Lambda^2 \|\sqrt{\rho_0+\sigma^{\delta}(t)}\vec u^d(t)\|_{L^2(\Omega)}^2 + C\delta^3 e^{\lambda_1 t} \\
&\quad+ \frac32 \Lambda^2 \int_0^t\Big( \mu\|\nabla \vec u^d(s)\|_{L^2(\Omega)}^2  - \int_{2\pi L\bT}(\xi_+ |u_1^d(s,x_1,1)|^2 +\xi_- |u_1^d(s, x_1,-1)|^2)  dx_1\Big)ds\\
&\quad+ \frac32  \int_0^t \int_{2\pi L\bT} ( \xi_+ |\Lambda u_1^d- \partial_t u_1 ^d|^2(s,x_1,1) +\xi_- |\Lambda u_{1}^d- \partial_t u_1^d|^2(s, x_1,-1))  dx_1ds.
\end{split}
\end{equation}
Integrating \eqref{Ine135} in time from 0 to $t$ and using \eqref{Ine134} and Young's inequality,  the authors deduce \eqref{Ine137} without providing any detailed explanation.

However, we  observe by integrating \eqref{Ine135} in time  that 
\[
\begin{split}
\|\sqrt{\rho_0 +\sigma^{\delta}(t)}\vec u^d(t)\|_{L^2(\Omega)}^2  \leq \frac1{\Lambda} \int_0^t e^{\Lambda(t-s)} \|\partial_t \vec u^d(s)\|_{L^2(\Omega)}^2 ds + Ce^{3\lambda_1 t}. 
\end{split}
\]
Then the l.h.s of \eqref{Ine137} will be bounded by 
\begin{equation}\label{Ine137_1}
\begin{split}
&\|\sqrt{\rho_0+\sigma^{\delta}(t)} \partial_t \vec u^d(t)\|_{L^2(\Omega)}^2 +\Lambda\mu \|\nabla\vec u^d(t)\|_{L^2(\Omega)}^2+ \mu\int_0^t \|\nabla\partial_t \vec u^d(s)\|_{L^2(\Omega)}^2  ds \\
&\leq \Lambda \int_0^t e^{\Lambda(t-s)} \|\partial_t \vec u^d(s)\|_{L^2(\Omega)}^2 ds  + \frac12 \Lambda\mu\|\nabla\vec u^d(t)\|_{L^2(\Omega)}^2 \\
&\quad+ \frac12\mu \int_0^t \|\nabla \partial_t \vec u^d(s)\|_{L^2(\Omega)}^2 ds + \frac32 \Lambda^2\mu \int_0^t\|\nabla \vec u^d(s)\|_{L^2(\Omega)}^2 ds\\
&\quad+ \Lambda \int_{2\pi L\bT} ( \xi_+ |u_1^d(t,x_1,1)|^2 +\xi_- |u_1^d(t, x_1,-1)|^2)  dx_1 \\
&\quad-\Lambda^2  \int_0^t \int_{2\pi L\bT}(\xi_+ |u_1^d(s,x_1,1)|^2 +\xi_- |u_1^d(s, x_1,-1)|^2)  dx_1ds\\
&\quad+ \int_0^t \int_{2\pi L\bT} ( \xi_+ |\Lambda u_{1}^d- \partial_tu_1^d|^2(s,x_1,1) +\xi_- |\Lambda u_{1}^d- \partial_t u_1^d|^2(s, x_1,-1))  dx_1+ C\delta^3e^{3\lambda_1t}.
\end{split}
\end{equation}
We are not clear about the way in \cite{DJL20} to remove all integral terms over $2\pi L\bT$ in the r.h.s of \eqref{Ine137_1} to get \eqref{Ine137} for all $\mu>0$, especially the following term 
\[
 \int_0^t \int_{2\pi L\bT} ( \xi_+ |\Lambda u_{1}^d- \partial_tu_1^d|^2(s,x_1,1) +\xi_- |\Lambda u_{1}^d- \partial_t u_1^d|^2(s, x_1,-1))  dx_1.
 \]

\end{document}